\documentclass[a4paper,12pt,twoside]{amsart}
\usepackage{amsmath}
\usepackage{amsthm}
\usepackage{amssymb}
\usepackage{hyperref}
\usepackage{rotating}
\usepackage{amssymb}
\usepackage{pifont}
\usepackage{amsmath}
\usepackage{enumerate}
\usepackage{graphicx}
\usepackage[all]{xy}
\usepackage{tikz}
\usepackage{hyperref}
\usetikzlibrary{arrows,matrix}

\newtheorem{theorem}{Theorem}[section]
\newtheorem{lemma}[theorem]{Lemma}
\newtheorem{proposition}[theorem]{Proposition}
\newtheorem{corollary}[theorem]{Corollary}

\newtheorem{definition}[theorem]{Definition}

\newtheorem{qn}[theorem]{Question}
\theoremstyle{remark}
\newtheorem{remark}[theorem]{Remark}
\newtheorem{example}[theorem]{Example}
\numberwithin{equation}{section}

\newcommand{\ZZ}{\mathbb{Z}}
\newcommand{\QQ}{\mathbb{Q}}

\newcommand{\CC}{\mathbb{C}}

\begin{document}
\title{Bounding the volumes of singular Fano threefolds}
\author[Ching-Jui\ Lai]{Ching-Jui Lai}
\begin{abstract} Let $(X,\Delta)$ be an $n$-dimensional $\epsilon$-klt log $\QQ$-Fano pair. 
We give an upper bound for the volume ${\rm Vol}(-(K_X+\Delta))=(-(K_X+\Delta))^n$ when $n=2$ 
or $n=3$ and $X$ is \mbox{$\QQ$-factorial} of $\rho(X)=1$. This bound is essentially sharp 
for $n=2$. Existence of an upper bound for anticanonical volumes is related the Borisov-Alexeev-Borisov 
Conjecture which asserts boundedness of the set of $\epsilon$-klt log $\QQ$-Fano varieties 
of a given dimension $n$. 
\end{abstract}
\maketitle

Throughout this article, we work over field of complex numbers $\CC$. We recall the definition 
of singularities of pairs and log $\QQ$-Fano pairs. 

\begin{definition}\label{epsilon} A pair $(X,\Delta)$ consists of a normal projective variety 
$X$ and a boundary $\Delta$, i.e., a $\QQ$-divisor $\Delta$ with coefficients in $[0,1]$, such that 
$K_X+\Delta$ is $\QQ$-Cartier. Let $\pi:Y\rightarrow X$ be a log resolution 
of $(X,\Delta)$, the discrepancy $a(E,X,\Delta)$ of a divisor $E$ on $Y$ with respect to the 
pair $(X,\Delta)$ is defined by \mbox{$a(E,X,\Delta)={\rm mult}_E(K_Y-\pi^*(K_X+\Delta))$}. 
We say that $(X,\Delta)$ has only terminal (resp. canonical) singularities if $a(E,X,\Delta)>0$ 
(resp. $\geq0$) for any $\pi$-exceptional divisor $E$ on $Y$. We say that $(X,\Delta)$ 
is klt (resp. $\epsilon$-klt for some $0<\epsilon<1$) if $a(E,X,\Delta)>-1$ (resp. $>-1+\epsilon$) 
for any divisor $E$ on $Y$. Note that smaller $\epsilon$ corresponds to worse singularities. 

A pair $(X,\Delta)$ is (weak) log $\QQ$-Fano if the $\QQ$-Cartier divisor $-(K_X+\Delta)$ 
is ample (resp. nef and big). 
\end{definition}

For a klt pair $(X,\Delta)$ with $\kappa(K_X+\Delta)=-\infty$, according to the log minimal model program, 
there exists a birational map $\phi:X\dashrightarrow Y$ and a morphism $Y\rightarrow Z$ 
such that for $\Delta'=\phi_*\Delta$, the pair $(Y_z,\Delta'_z)$ is log $\QQ$-Fano 
with $\rho(Y_z)=1$ for general $z\in Z$. In particular, log $\QQ$-Fano pairs are the building blocks 
for pairs with negative Kodaira dimension. It is also expected that the set of mildly 
singular $\QQ$-Fano varieties is bounded. 

\begin{definition}\label{bd} We say that a collection of varieties $\{X_\lambda\}_{\lambda\in\Lambda}$ 
is bounded if there exists $h:\mathcal{X}\rightarrow S$ a morphism of finite type of Noetherian 
schemes such that for each $X_\lambda$, $X_\lambda\cong\mathcal{X}_s$ for some $s\in S$. 	
\end{definition}

For example, the set of all the $n$-dimensional smooth Fano manifolds is bounded by \cite{KMMc}. 
Boundedness is also known for terminal $\QQ$-Fano $\QQ$-factorial threefolds 
of Picard number one by \cite{Kaw:fano} and for canonical $\QQ$-Fano threefolds by \cite{KMMT}. 
However, if one considers the set of all klt $\QQ$-Fano varieties with Picard number one of 
a given dimension, \cite{Lin} and \cite{Oka:fano} have shown that birational boundedness fails. 
The problem is that the category of klt singularities is too big to be bounded since, for example, it contains 
finite quotients of arbitrarily large order. To get boundedness, one restricts to a smaller 
class of singularities, known as $\epsilon$-klt singularities. Precisely we have the following 
conjecture due to A. Borisov, L. Borisov, and V. Alexeev, which is still open in dimension three 
and higher.

\theoremstyle{conj}
\newtheorem*{conj}{\textbf{Borisov-Alexeev-Borisov Conjecture}}
\begin{conj}\label{BAB} Fix $0<\epsilon<1$, an integer $n>0$, and consider the set of all 
$n$-dimensional $\epsilon$-klt log $\QQ$-Fano pairs $(X,\Delta)$. The set of underlying 
varieties $\{X\}$ is bounded. 
\end{conj} 

A. Borisov and L. Borisov establish the B-A-B Conjecture for toric varieties in \cite{B93:tor}. 
V. Alexeev establishes the two dimensional B-A-B Conjecture in \cite{Ale:bs} with a simplified 
argument given in \cite{AM:bs}. Our original motivation for studying the B-A-B Conjecture is 
that it is related to the conjectural termination of flips in the minimal model program. According 
to \cite{BSh:acc}, the log minimal model program, the a.c.c.\footnote{An a.c.c. (respectively d.c.c.) 
set is a set of real numbers satisfying the ascending (descending) chain condition, i.e., it 
contains no infinite strictly increasing (decreasing) sequences.} for minimal log discrepancies, 
and the B-A-B Conjecture in dimension $\leq d$ implies termination of log flips in dimension $\leq d+1$ 
for effective pairs. 

The following questions concerning log $\QQ$-Fano pairs $(X,\Delta)$ are relevant to the B-A-B 
Conjecture: 
\begin{enumerate}
  \item[$(i)$] The Cartier index of $K_X+\Delta$ of an $n$-dimensional $\epsilon$-klt 
               log $\QQ$-Fano pair $(X,\Delta)$ is bounded from above by a fixed integer 
               $r(n,\epsilon)$ depending only on $n=\dim X$ and $\epsilon$;
  \item[$(ii)$] The volume ${\rm Vol}(-(K_X+\Delta))=(-(K_X+\Delta))^n$ of an $n$-dimensional $\epsilon$-klt 
		log $\QQ$-Fano pair $(X,\Delta)$ is bounded from above by a fixed integer $M(n,\epsilon)$ 
		depending only on \mbox{$n=\dim X$} and $\epsilon$;
  \item[$(iii)$] (\textbf{Batyrev Conjecture}) For given positive integers $n$ and $r$, consider the set 
		 of all $n$-dimensional klt log $\QQ$-Fano pairs $(X,\Delta)$ with $r(K_X+\Delta)$ 
		 a Cartier divisor. The set of underlying varieties $\{X\}$ is bounded. 
\end{enumerate}
It is clear that the B-A-B Conjecture follows from $(i)$ and $(iii)$. Note that recently C. Hacon, 
J. M$^{\rm c}$Kernan, and C. Xu have announced a proof of the Batyrev Conjecture $(iii)$. In 
general it is very hard to establish $(i)$. Ambro in \cite{AF:tor} has proved $(i)$ for toric 
singularities when the boundaries have standard coefficients $\{1-\frac{1}{\ell}|\ell\in\ZZ_{\geq1}\}\cup\{1\}$. 
A necessary condition for $(i)$ to hold is that we need to restrict the coefficients of boundaries 
to be in a fixed d.c.c. set. A counterexample for the general statement is given by the set of 
pairs $(\mathbb{P}^1,\frac{1}{N}\{{\rm pt}\})$ for $N\geq1$. 

For the convenience of the reader, we include a well-known argument (to the experts) establishing 
the B-A-B Conjecture via condition $(i)$ and $(ii)$ in the cases $\Delta=0$ or $\rho(X)=1$. 

\begin{proposition}\label{bdd} Suppose that $\Delta=0$ or $\rho(X)=1$, then the B-A-B 
Conjecture holds if both $(i)$ and $(ii)$ above are true. 
\end{proposition}
\begin{proof} Suppose that $\Delta=0$ and let $X$ be any $\epsilon$-klt $\QQ$-Fano variety 
of dimension $n$. The following statements together imply the B-A-B conjecture in this case: 
    \begin{enumerate}
      \item The divisor $N(-K_X)$ is a very ample line bundle for a fixed 
            $N$ depending only on $n$ and $\epsilon$;
      \item The set of Hilbert polynomials $\mathfrak{F}=\{P(t)=\chi(\mathcal{O}_X(-NK_X)^{\otimes t})\}$ 
            associated to all $n$-dimensional $\epsilon$-klt $\QQ$-Fano varieties 
            is finite.
    \end{enumerate}
Indeed, statements $(1)$ and $(2)$ imply that the set of $n$-dimensional $\epsilon$-klt 
$\QQ$-Fano varieties is contained in a finite union of Hilbert schemes $\coprod_{P(t)\in\mathfrak{F}}\mathcal{H}_{P(t)}$, 
where each $\mathcal{H}_{P(t)}$ is Noetherian. 

From $(i)$, there is an upper bound $r(n,\epsilon)$ of the Cartier index of $K_X$ depending 
only on $n$ and $\epsilon$. It follows that $rK_X$ is a line bundle for $r=r(n,\epsilon)$. 
By \cite{Kol:ebf}, $|-mrK_X|$ is base point free for any $m>0$ divisible by a constant $N_1(n)>0$ 
depending only on $n=\dim X$. Since $|-mrK_X|$ is ample and base point free for $m>0$ sufficiently 
divisible, it defines a finite morphism. By \cite[Theorem 5.9]{Kol:sop}, the map induced by 
$|-lrK_X|$ is birational for any $l>0$ divisible by a constant $N_2(n)>0$ depending only on 
$n=\dim X$. Since a finite birational morphism of normal varieties is an isomorphism, it 
follows that there exists an effective embedding by $|M(-rK_X)|$ for some fixed $M>0$ depending 
only on $n=\dim X$. Take $N=Mr$, we have $(1)$. 

By \cite{KMa:RR}, the coefficients of the Hilbert polynomial $P(t)=h^0(\mathcal{O}_X(tH))$ 
of a polarized variety $(X,H)$ with $H$ an ample line bundle can be bounded by the intersection 
numbers $|H^n|$ and $|H^{n-1}.K_X|$. Since by $(i)$ there exists an integer $r=r(n,\epsilon)>0$ depending 
only on $n=\dim X$ and $\epsilon$ such that $-rK_X$ is an ample line bundle, set $H=-rK_X$ 
and apply $(ii)$. It follows that there are only finitely many Hilbert polynomials for the 
set of anti-canonically polarized $\epsilon$-klt Fano varieties $\{(X,-rK_X)\}$. 

If $\rho(X)=1$, then $-(K_X+\Delta)$ being ample implies that $-K_X$ is also ample. 
It is clear that $X$ is also $\epsilon$-klt and hence boundedness follows from the 
same proof as above.
\end{proof}

An effective upper bound in $(ii)$ is obtained for smooth Fano $n$-folds in \cite{KMMc} 
and for canonical $\QQ$-Fano threefolds in \cite{KMMT}. In this paper, we obtain an 
effective answer to question $(ii)$ in dimension two, i.e., for log del Pezzo surfaces.  

\theoremstyle{thm}
\newtheorem*{thm}{Theorem A}
\begin{thm}\label{0.6}{(Theorem \ref{I})} Let $(X,\Delta)$ be an $\epsilon$-klt 
weak log del Pezzo surface. The volume ${\rm Vol}(-(K_X+\Delta))=(K_X+\Delta)^2$ satisfies 
    \begin{align*} (K_X+\Delta)^2\leq\max\{64,\frac{8}{\epsilon}+4\}.  
    \end{align*}
Moreover, this upper bound is in a sharp form: There exists a sequence 
of $\epsilon$-klt del Pezzo surfaces whose volume grows linearly with 
respect to $1/\epsilon$. 
\end{thm}

Let $(X,\Delta)$ be an $\epsilon$-klt weak log del Pezzo surface and 
$X_{\rm min}$ be the minimal resolution of $(X,\Delta)$. Alexeev and 
Mori have shown in \cite[Theorem 1.8]{AM:bs} that $\rho(X_{\rm min})\leq128/\epsilon^5$. 
Also from \cite[Lemma 1.2]{AM:bs} (or see the proof of Theorem \ref{I}), 
an exceptional curve $E$ on $X_{\rm min}$ over $X$ has degree $1\leq-E^2\leq2/\epsilon$.  
When $\Delta=0$, since the Cartier index of $K_X$ is bounded from 
above by the determinant of the intersection matrix $(E_i.E_j)$ of the 
exceptional curves $E_i$'s on $X_{\rm min}$ over $X$, it follows that 
the Cartier index bound $r(2,\epsilon)$ in the statement $(i)$ satisfies 
    \begin{equation*} r(2,\epsilon)\leq 2(2/\epsilon)^{128/\epsilon^5}. \tag{$\Diamond$}
    \end{equation*}
An upper bound of $(K_X+\Delta)^2$ is implicitly mentioned in \cite{Ale:bs} 
but not clearly written down. It is also not clear if the upper bound 
$(\Diamond)$ is optimal. In view of Theorem A, this seems unlikely.

As a second result, we also obtain an upper bound of the volumes for $\epsilon$-klt 
$\QQ$-factorial log $\QQ$-Fano threefolds of Picard number one. Recall that a 
variety $X$ is $\QQ$-factorial if each Weil divisor is $\QQ$-Cartier. 

\theoremstyle{thm}
\newtheorem*{thm'}{Theorem B}
\begin{thm'}\label{0.8}{$($Theorem \ref{II}$)$} Let $(X,\Delta)$ be 
an $\epsilon$-klt $\QQ$-factorial log $\QQ$-Fano threefold of $\rho(X)=1$. The 
degree $-K_X^3$ satisfies
    \begin{align*} -K_X^3\leq(\frac{24M(2,\epsilon)R(2,\epsilon)}{\epsilon}+12)^3, 
    \end{align*}
where $R(2,\epsilon)$ is an upper bound of the Cartier index of $K_S$ for $S$ any $\epsilon/2$-klt log del Pezzo surface of $\rho(S)=1$ 
and $M(2,\epsilon)$ is an upper bound of the volume ${\rm Vol}(-K_S)=K_S^2$ for $S$ any $\epsilon/2$-klt log del Pezzo surface of $\rho(S)=1$. 
Note that $M(2,\epsilon)\leq\max\{64,16/\epsilon+4\}$ from Theorem A and $R(2,\epsilon)\leq2(4/\epsilon)^{128\cdot2^5/\epsilon^5}$ from $(\Diamond)$.
\end{thm'}

For a $\QQ$-factorial $\epsilon$-klt log $\QQ$-Fano pair $(X,\Delta)$ of 
$\rho(X)=1$, since $-(K_X+\Delta)^3\leq-K_X^3$ and $X$ is also $\epsilon$-klt, 
by Theorem B we get an upper bound of the anticanonical volume ${\rm Vol}(-(K_X+\Delta))=-(K_X+\Delta)^3$. 
However, it is not expected that the bound in Theorem B is sharp or in a sharp form.

Note that $\QQ$-factoriality is a technical assumption. However, this condition 
is natural in the sense that starting from a smooth variety, each variety 
constructed by a step of the minimal model program remains $\QQ$-factorial. In 
dimension two, normal surfaces with rational singularities, e.g., klt singularities, 
are always $\QQ$-factorial. 

Instead of using deformation theory of rational curves as in \cite{KMMT}, 
the Riemann-Roch formula as in \cite{Kaw:fano}, or the sandwich argument 
of \cite{Ale:bs}, we aim to create isolated non-klt centers by the method 
developed in \cite{Mck:fano}. The point is that deformation theory for rational 
curves on klt varieties is much harder and so far no effective Riemann-Roch 
formula is known for klt threefolds. 

The rest of this paper is organized as follows: In Section 1, we study non-klt centers. 
In Section 2, we illustrate the general method in \cite{Mck:fano} for obtaining an upper 
bound of anticanonical volumes in Theorem A and B. In Section 3, we review the theory of 
families of non-klt centers in \cite{Mck:fano}. In Section 4, we study weak log del Pezzo 
surfaces and prove Theorem A. In Section 5, we prove Theorem B. 

{\small\noindent {\bf Acknowledgment.} The author is grateful to Professor 
Christopher Hacon, Professor James M$^{\rm c}$Kernan, and Professor Chenyang 
Xu for many useful discussions and suggestions.

\section{Non-klt Centers}
For the theory of the singularities in the minimal model program, we refer to \cite{KM:bgav}.

\begin{definition} Let $(X,\Delta)$ be a pair. A subvariety $V\subseteq X$ is called 
a \textbf{non-klt center} if it is the image of a divisor of discrepancy at most $-1$. 
A \textbf{non-klt place} is a valuation corresponding to a divisor of discrepancy at 
most $-1$. The \textbf{non-klt locus} ${\rm Nklt}(X,\Delta)\subseteq X$ is the union 
of the non-klt centers. If there is a unique non-klt place lying over the generic point 
of a non-klt center $V$, then we say that $V$ is \textbf{exceptional}. If $(X,\Delta)$ 
is not klt along the generic point of a non-klt center $V$, then we say that $V$ is 
\textbf{pure}.
\end{definition}
The non-klt places/centers here are the log canonical (lc) places/centers 
in \cite{Mck:fano}. 

A standard way of creating a non-klt center on an $n$-dimensional variety $X$ is to find a 
very singular divisor: Fix $p\in X$ a smooth point, if $\Delta$ is a $\QQ$-Cartier divisor on 
$X$ with ${\rm mult}_p\Delta\geq n$, then $p\in {\rm Nklt}(X,\Delta)$. Indeed, consider the 
blow up $\pi:Y={\rm Bl}_pX\rightarrow X$ and let $E$ be the unique exceptional divisor 
with $\pi(E)=p$, then the discrepancy 
\begin{align*} a(E,X,\Delta)={\rm mult}_E(K_Y-\pi^*(K_X+\Delta))=(n-1)-{\rm mult}_E(\pi^*(\Delta))\leq -1,  
\end{align*}
as $n-1={\rm mult}_E(K_Y-\pi^*K_X)$ and ${\rm mult}_E(\pi^*\Delta)={\rm mult}_p\Delta\geq n$. 

We can find singular divisors by the following lemma.

\begin{lemma}\label{sd} Let $X$ be an $n$-dimensional complete complex variety and $D$ 
be a divisor with $h^i(X,\mathcal{O}(mD))=O(m^{n-1})$ for all $i>0$, e.g., $D$ is big and 
nef. Fix a positive rational number $\alpha$ with $0<\alpha^n<D^n$. For $m\gg 0$ 
and any $x\in X_{\rm sm}$, there exists a divisor $E_x\in|mD|$ with ${\rm mult}_x(E_x)\geq m\cdot\alpha.$
\end{lemma}
\begin{proof} This is \cite[Proposition 1.1.31]{Laz:p1}.
\end{proof}

We will apply Lemma \ref{sd} to the case where $(X,\Delta)$ is an $n$-dimensional log 
$\QQ$-Fano pair: Write $(-(K_X+\Delta))^n>(\omega n)^n$ for some rational number $\omega>0$, 
then as $h^i(X,\mathcal{O}(-m(K_X+\Delta)))=0$ for $m>0$ sufficiently divisible by the Kawamata-Viehweg 
vanishing theorem, we can find for each $p\in X_{\rm sm}$ a $\QQ$-divisor $\Delta_p$ such 
that $\Delta_p\sim_\QQ-(K_X+\Delta)/\omega$ and ${\rm mult}_p(\Delta_p)\geq n$. In particular, 
$p\in{\rm Nklt}(X,\Delta+\Delta_p)$.

The non-klt centers satisfy the following Connectedness Lemma of Koll\'ar and Shokurov, which 
is simply a formal consequence of the Kawamata-Viehweg vanishing theorem and is 
the most important ingredient in this paper. 
\begin{lemma}\label{con} Let $(X,\Delta)$ be a log pair. Let $f:X\rightarrow Z$ be a 
projective morphism with connected fibers such that the image of every component of $\Delta$ 
with negative coefficient is of codimension at least two in $Z$. If $-(K_X+\Delta)$ 
is big and nef over $Z$, then the intersection of ${\rm Nklt}(X,\Delta)$ with each fiber 
$X_z=f^{-1}(z)$ is connected. 
\end{lemma}
\begin{proof} For simplicity, we assume that $Z={\rm Spec}(\CC)$ is a point and $(X,\Delta)$ 
is log smooth, i.e., $X$ is smooth and $\Delta$ has simple normal crossing support. Then 
the identity map ${\rm id}_X:X\rightarrow X$ is a log resolution of $(X,\Delta)$ and 
${\rm Nklt}(X,\Delta)=\llcorner\Delta\lrcorner$. Consider the exact sequence 
\begin{align*} \cdots\rightarrow H^0(X,\mathcal{O}_X)\rightarrow H^0(X,\mathcal{O}_{\llcorner\Delta\lrcorner})\rightarrow H^1(X,\mathcal{O}_X(-\llcorner\Delta\lrcorner))\rightarrow\cdots. 
\end{align*}
Since $-\llcorner\Delta\lrcorner=K_X+\{\Delta\}-(K_X+\Delta)$ and $(X,\{\Delta\})$ is klt, 
we have $H^1(X,\mathcal{O}_X(-\llcorner\Delta\lrcorner))=0$ by the Kawamata-Viehweg vanishing 
theorem as $-(K_X+\Delta)$ is nef and big. Since $H^0(X,\mathcal{O}_X)\cong\CC$, we see that 
${\rm Nklt}(X,\Delta)=\llcorner\Delta\lrcorner$ is connected. 

For the general case, see \cite[Theorem 17.4]{corti:flip34}. 
\end{proof}

Here is an example showing that $-(K_X+\Delta)$ being nef and big is necessary in 
the Connectedness Lemma \ref{con}. 
\begin{example} Let $X$ be $\mathbb{P}^1\times\mathbb{P}^1$ and denote by $F$ (resp. $G$) 
the fiber of the first (resp. second) projection to $\mathbb{P}^1$. Consider 
$\Delta_1=F_1+F_2$ the sum of two distinct fibers of the first projection to $\mathbb{P}^1$ 
and $\Delta_2=F+G$ the sum of two fibers with respect to the two different projections to 
$\mathbb{P}^1$. Then ${\rm Nklt}(X,\Delta_1)=F_1+F_2$ is not connected while 
${\rm Nklt}(X,\Delta_2)=F+G$ is connected. Note that $-(K_X+\Delta_1)$ is nef but not 
big while $-(K_X+\Delta_2)$ is nef and big. 
\end{example}

Later on, we will produce not only non-klt centers but \emph{isolated} non-klt centers. 
The following theorem is the main technique which allows us to cut down the dimension 
of non-klt centers. 
\begin{theorem}\label{cut}$($\cite[Theorem 6.8.1]{Kol:sop}$)$ Let $(X,\Delta)$ be klt, projective and $x\in X$ a 
closed point. Let $D$ be an effective $\QQ$-Cartier $\QQ$-divisor on $X$ such 
that $(X,\Delta+D)$ is log canonical in a neighborhood of $x$. Assume that 
${\rm Nklt}(X,\Delta+D)=Z\cup Z'$ where $Z$ is irreducible, $x\in Z$, and 
$x\notin Z'$. Set $k=\dim Z$. If $H$ is an ample $\QQ$-divisor on $X$ such 
that $(H^k.Z)>k^k$, then there is an effective $\QQ$-divisor $B\equiv H$ and 
rational numbers $1\gg\delta>0$ and $0<c<1$ such that
    \begin{enumerate}
      \item[(1)] $(X,\Delta+(1-\delta)D+cB)$ is non-klt in a neighborhood of $x$, and
      \item[(2)] ${\rm Nklt}(X,\Delta+(1-\delta)D+cB)=W\cup W'$ where $W$ is 
		 irreducible, $x\in W$, $x\notin W'$ and $\dim W<\dim Z$.
    \end{enumerate}
\end{theorem}

\section{A Guiding Example}
The idea in \cite{Mck:fano} for obtaining an upper bound for the anticanonical 
volumes is to create isolated non-klt centers and then use the Connectedness 
Lemma \ref{con}: For simplicity, we assume that $\Delta=0$. Write $(-K_X)^n>(\omega n)^n$ 
for a positive rational number $\omega$. For each $p\in X_{\rm sm}$, we can 
find an effective $\QQ$-divisor $\Delta_p\sim_\QQ-K_X/\omega$ such that ${\rm mult}_p\Delta_p\geq n$ 
and hence $p\in{\rm Nklt}(X,\Delta_p)$. The observation is that if $\omega\gg0$, 
then for general $p\in X$, $p\in{\rm Nklt}(X,\Delta_p)$ can not be an isolated point. 
Indeed, if this is not true, then for two general points $p,q\in X$, the set 
${\rm Nklt}(X,\Delta_p+\Delta_q)$ would contain $\{p,q\}$ as isolated non-klt centers. 
But the divisor $K_X+\Delta_p+\Delta_q\sim_\QQ(1-\frac{2}{\omega})(-K_X)$ 
is nef and big for $\omega>2$. By the Connectedness Lemma \ref{con}, ${\rm Nklt}(X,\Delta_p+\Delta_q)$ 
must be connected; a contradiction. 

Therefore, for general $p\in X$ the minimal non-klt center $V_p\subseteq{\rm Nklt}(X,\Delta_p)$ 
passing through $p$ is typically positive dimensional. We would like to show 
that the restricted volume ${\rm Vol}(-K_X|_{V_p})$ on the minimal non-klt center 
$V_p$ is large when $\omega\gg0$. Hence, we can cut down the dimension of non-klt 
centers by Theorem \ref{cut}. After doing this finitely many times, we get 
isolated non-klt centers and we are done. 

In general, it is hard to find a lower bound of the restricted volume ${\rm Vol}(-K_X|_{V_p})$ 
on the minimal non-klt center $V_p$. We illustrate M$^{\rm c}$Kernan's method by 
studying families of non-klt centers to obtain a lower bound of the restricted volumes 
on the non-klt center of an $\epsilon$-klt log $\QQ$-Fano variety via the following 
guiding example, cf. \cite{Mck:fano}. 

\begin{example}\label{guiding} Let $X$ be the projective cone over a rational normal 
curve of degree $d\geq2$ with the unique singular point $O\in X$. The blow up 
$\pi:Y={\rm Bl}_OX\rightarrow X$ is a resolution of $X$ where $Y$ is a 
$\mathbb{P}^1$-bundle $f:Y\rightarrow\mathbb{P}^1$ over $\mathbb{P}^1$:

\begin{equation*}
\begin{tikzpicture}
\node (A) at (0,1) {$X$};
\node (B) at (2,1) {$Y$};
\node (B') at (2.5,1) {$\supseteq$};
\node (C) at (2,-1) {$\mathbb{P}^1$};
\node (B') at (2.5,-1) {$\ni$};
\node (D) at (3,1) {$F_t$};
\node (D') at (3.75,1) {$\cong\mathbb{P}^1 $};
\node (E) at (3,-1) {$t.$};
\path[->] (B) edge node[above]{$\pi$}(A);
\path[->] (B) edge node[left]{$f$}(C);
\path[->] (D) edge node[above]{}(E);
\end{tikzpicture}
\end{equation*}

It is easy to show that
    \begin{enumerate}
	\item[(a)] $K_Y=\pi^*K_X+(-1+2/d)E$, where $E$ is the unique exceptional 
                   divisor and hence $X$ is $\epsilon$-klt for $\epsilon=1/d$;
	\item[(b)] $X$ is $\QQ$-factorial of Picard number one and $-K_X\sim_\QQ(d+2)l$ 
                   is an ample $\QQ$-Cartier divisor, where $l$ is the class of a ruling 
                   of $X$. Hence $X$ is an $\epsilon$-klt del Pezzo surface;
        \item[(c)] ${\rm Vol}(-K_X)=d+4+4/d$ is a linear function of $d=1/\epsilon$ 
                   and provides the required example in Theorem A.
    \end{enumerate}
    
Let $p\in X$ be a general point. Then $p$ is not the vertex $O$ and the unique 
ruling $l_p$ passing through $p$ is the non-klt center of the log pair $(X,l_p)$, 
i.e., $l_p={\rm Nklt}(X,l_p)$. Moreover, the proper transform $F_p$ of $l_p$ on 
$Y$ is a fiber of the $\mathbb{P}^1$-bundle $f:Y\rightarrow\mathbb{P}^1$. In this 
case, the $\mathbb{P}^1$-bundle structure of $Y$ is a covering family of non-klt 
centers of $X$ since the map $\pi:Y\rightarrow X$ is dominant. 

For $p,q\in X$ two general points, let $l_p$ and $l_q$ be the rulings passing 
through $p$ and $q$ respectively. Consider the pair $K_Y+(1-2/d)E=\pi^*K_X$. 
By the Connectedness Lemma \ref{con}, the non-klt locus ${\rm Nklt}(K_Y+(1-2/d)E+\pi^*(l_p+l_q))$ 
containing $F_p\cup F_q$ is connected as 
    \begin{align*}-(K_Y+(1-2/d)E+\pi^*(l_p+l_q))=-\pi^*(K_X+l_p+l_q)\equiv d\pi^*l
    \end{align*}
is nef and big. In fact, the fibers $F_p$ and $F_q$ are connected in 
${\rm Nklt}(K_Y+(1-2/d)E+\pi^*(l_p+l_q))$ by $E$ as 
    \begin{align*}F_p\cup F_q\subseteq{\rm Nklt}(K_Y+(1-2/d)E+\pi^*(l_p+l_q))\subseteq\pi^{-1}({\rm Nklt}(K_X+l_p+l_q))=F_p\cup F_q\cup E, 
    \end{align*}
where the second inclusion follows from the definition of non-klt centers. In particular, 
    \begin{align*} {\rm mult}_E(\pi^*(l_p+l_q))\geq\frac{2}{d}=2\epsilon.     
    \end{align*}
By symmetry, $\pi^*l_p$ must contribute multiplicity at least $1/d=\epsilon$ to 
the component $E$ (and in fact is exactly $1/d$ in this case), i.e., 
    \begin{equation}\label{00} \pi^*l_p\geq\epsilon E.     
    \end{equation}
Note that  
    \begin{equation}\label{01} l_p\sim_\QQ \frac{-K_X}{\sqrt{d\cdot{\rm Vol}(-K_X)}}. 
    \end{equation}
By intersecting both sides of \eqref{00} with a general fiber $F$ of 
$f:Y\rightarrow\mathbb{P}^1$, we get for the ruling $l=\pi_*(F)$, 
    \begin{equation}\label{02} \frac{1}{\sqrt{d\cdot{\rm Vol}(-K_X)}}{\rm deg}_{l}(-K_X)=\pi^*l_p.F\geq\epsilon E.F.    
    \end{equation}
Since $F$ is a general fiber meeting the horizontal divisor $E$ at a smooth 
point, $E.F\geq1$. (In this case $E.F=1$.) Combining all of these, we obtain a 
lower bound of the restricted volume ${\rm deg}_{l}(-K_X)$,
    \begin{align*} {\rm deg}_l(-K_X)\geq\epsilon{\sqrt{d\cdot{\rm Vol}(-K_X)}}.  
    \end{align*}
Note that since in this case ${\rm deg}_l(-K_X)=-K_X.l=-K_Y.\pi^*l\leq2$, it follows that 
${\rm Vol}(-K_X)=K_X^2\leq4d=4/\epsilon$. 
\end{example}

In summary, the method of getting an upper bound of the anticanonical volumes is to obtain 
a lower bound of the restricted volume ${\rm Vol}(-(K_X+\Delta)|_{V_p})$ on the non-klt 
centers $V_p$, which can be outlined in the following steps:
    \begin{itemize}
        \item Suppose that ${\rm Vol}(-(K_X+\Delta))=(-(K_X+\Delta))^n>(\omega n)^n$ 
	      for a positive rational number $\omega$. We will show that $\omega>0$ 
	      can not be arbitrarily large. 
	\item For general $p\in X$, choose 
              \begin{align*} \Delta_p\sim_\QQ\frac{-(K_X+\Delta)}{\omega}, 
	      \end{align*}
              so that $p\in{\rm Nklt}(X,\Delta+\Delta_p)$. Let $V_p\subseteq{\rm Nklt}(X,\Delta+\Delta_p)$ 
	      be the minimal non-klt center containing $p$. 
	\item Construct covering families of non-klt centers by ``lining up'' (part of the) 
	      non-klt centers $\{V_p\}$, see Section \ref{S2}. This is the generalization 
              of the $\mathbb{P}^1$-bundle structure in the Example \ref{guiding} and 
	      is called a \emph{covering families of tigers} in \cite{Mck:fano}.  
        \item Use the Connectedness Lemma \ref{con} to obtain a lower bound of the restricted 
	      volume 
	      \begin{align*} {\rm Vol}(-(K_X+\Delta)|_{V_p}))=(-(K_X+\Delta)|_{V_p}))^{\dim V_p},
	      \end{align*} 
	      on the non-klt center $V_p$ in terms of $\omega$ and $\epsilon$. This 
              is the most technical part.
        \item If $\omega\gg0$, then we cut down the dimension of non-klt centers by 
              Theorem \ref{cut}. After finitely many steps, we get isolated non-klt 
              centers and hence a contradiction to the Connectedness Lemma \ref{con}.
    \end{itemize}
The difficulty of this argument arises in dimension three in many places. 
First of all, the non-klt centers can be of dimension one or two and we have 
to deal with them case by case. When we have one dimensional covering families 
of tigers, it is subtle to detect the contribution of the \mbox{$\epsilon$-klt} 
condition from some horizontal subvariety, which is analogous to the exceptional 
curve $E$ in Example \ref{guiding}. This is done by applying a differentiation 
argument to construct a better behaved covering family of tigers, see \ref{IV.1}. 
In case we have two dimensional non-klt centers, complications arise for computing 
intersection numbers as the total space $Y$ of a covering family of tigers is in 
general not $\QQ$-factorial. This can be fixed by replacing $Y$ with a suitable 
birational model. To finish the proof, we also need to run a relative minimal model 
on the covering family of tigers and study the geometry of all possible outcomes. 

\section{Covering Families of Tigers}\label{S2}
The main reference for this section is \cite{Mck:fano}. 

\begin{definition}{$($\cite[Definition 3.1]{Mck:fano}$)$} Let $(X,\Delta)$ be a log pair 
with $X$ projective and $D$ a $\QQ$-Cartier divisor. We say that pairs of the form 
$(\Delta_t,V_t)$ form a \textbf{covering family of tigers} of dimension $k$ and weight 
$\omega$ if all of the following hold:
    \begin{enumerate} 
	\item there is a projective morphism $f:Y\rightarrow B$ of normal projective 
	      varieties such that the general fiber of $f$ over $t\in B$ is $V_t$;
	\item there is a morphism of $B$ to the Hilbert scheme of $X$ such that $B$ is 
	      the normalization of its image and $f$ is obtained by taking the normalization 
	      of the universal family;
	\item if $\pi:Y\rightarrow X$ is the natural morphism, then $\pi(V_t)$ is a minimal 
	      pure non-klt center of $K_X+\Delta+\Delta_t$;
	\item $\pi$ is generically finite and dominant;
	\item $\Delta_t\sim_\QQ D/\omega$, where $\Delta_t$ is effective;
	\item the dimension of $V_t$ is $k$. 
    \end{enumerate}  
\end{definition}
Note that by definition $k\leq \dim X-1$ and $\pi|_{V_t}:V_t\rightarrow \pi(V_t)$ is 
\emph{finite} and \emph{birational}. The covering family of tigers is illustrated in the 
following diagram:
\begin{center}
\begin{tikzpicture}
\node (A) at (-2,1) {$X$};
\node (B) at (0,1){ $Y$};
\node (C) at (0,-1){$B$};
\node (D) at (0.5,1){$\supseteq$};
\node (E) at (0.5,-1){$\ni$};
\node (F) at (1,1){$V_t$};
\node (G) at (1,-1){$t.$};
\path[->] (B) edge node[above]{$\pi$}(A);
\path[->] (B) edge node[left]{$f$}(C);
\path[->] (F) edge node[left]{}(G);
\end{tikzpicture}
\end{center}
We will sometimes also refer to $V_t$ as the minimal non-klt center of $(X,\Delta+\Delta_t)$. 

For $(X,\Delta)$ a log $\QQ$-Fano variety, we will always assume that 
$D=-\lambda(K_X+\Delta)$ for some $\lambda>0$. In particular, $D$ is assumed to be 
big and semi-ample. 

The existence of a covering family of tigers is achieved by constructing non-klt centers 
at general points of $X$ and then fitting a sub-collection of them into a fiber space. 
In order to fit the non-klt centers into a family, we use exceptional non-klt 
centers so that we patch up the unique non-klt place associated to each of them. The 
following lemma allows us to create exceptional non-klt centers. 
\begin{lemma}\label{exc} Let $(X,\Delta)$ be a log pair and let $D$ be a big and semi-ample 
$\QQ$-Cartier divisor. Write $D^n>(\omega n)^n$ for some positive rational number $\omega$. 
In order to find an upper bound of $\omega$ and hence an upper bound of ${\rm Vol}(D)=D^n$, 
for every $p\in X_{\rm sm}$ we may assume that there is a divisor $\Delta_p\sim_\QQ D/\omega$ 
such that the unique minimal non-klt center $V_p\subseteq{\rm Nklt}(X,\Delta+\Delta_p)$ containing 
$p$ is exceptional. 
\end{lemma}
\begin{proof} By Lemma \ref{sd}, for any $p\in X_{\rm sm}$ we can find an effective divisor 
$\Delta'_p\sim_{\QQ}\frac{D}{\omega}$ such that ${\rm mult}_p\Delta'_p\geq n$ and hence 
$p\in{\rm Nklt}(X,\Delta+\Delta'_p)$.

Fix $p\in X_{\rm sm}$, pick $0<\delta_p\leq1$ the unique rational number 
such that $(X,\Delta+\delta_p\Delta'_p)$ is log canonical but not klt at $p$. By 
\cite[Proposition 3.2, Lemma 3.4]{AF:lcc}, we can find an effective divisor $M_p\sim_\QQ D$ 
and some rational number $a>0$ such that for any rational number $0<\mu<1$, the 
pair $(X,(1-\mu)(\Delta+\delta_p\Delta'_p)+\mu\Delta+\mu aM_p)$ has a unique 
minimal non-klt center $V_p$ passing through $p$ which is exceptional. If we write 
    \begin{align*} \Delta_p:=(1-\mu)\delta_p\Delta'_p+\mu aM_p\sim_\QQ\frac{1}{\omega'_p}D, 
    \end{align*}
then
    \begin{align*} \omega'_p=\frac{\omega}{(1-\mu)\delta_p+\mu a\omega},   
    \end{align*}
and $(1-\mu)\delta_p+\mu a\omega<1+1/n$ for any $n\geq1$ if we pick $0<\mu\ll1$ sufficiently 
small. Hence $\omega'_p>\omega/(1+1/n)$. Since $D$ is semi-ample, by adding a small multiple of 
$D$ to $\Delta_p$ we have $\Delta_p\sim_\QQ D/\omega_n$ for $\omega_n=\omega/(1+2/n)$, and 
$(X,\Delta+\Delta_p)$ has a unique minimal non-klt center $V_p$ passing through $p$ which is 
exceptional. If there exists an upper bound of $\omega_n$ independent of $n$, then by taking 
$n\rightarrow\infty$, we get the same upper bound of $\omega$. 
\end{proof}

The following proposition is the construction of the covering family of tigers, see 
\cite[Lemma 3.2]{Mck:fano} or \cite[Lemma 3.2]{Tod}. 
\begin{proposition}\label{tiger} Let $(X,\Delta)$ and $\Delta_p$ be the same as in 
Lemma \ref{exc}. Then there exists a covering family of tigers $\pi:Y\rightarrow X$ 
of weight $\omega$ with $V_p\subseteq{\rm Nklt}(X,\Delta+\Delta_p)$ the unique 
minimal non-klt center passing through $p$. 
\end{proposition}
\begin{proof} Choose $m>0$ an integer such that $mD/\omega$ is integral and Cartier 
and let $B$ be the Zariski closure of points $\{m\Delta_p|p\in X_{\rm sm}\}\in|mD/\omega|$. 
Replace $B$ by an irreducible component which contains an uncountable subset $Q$ of 
$B$ such that the set $\{p\in X|\Delta_p\in Q\}$ is dense in $X$. This is possible 
since the $\Delta_p$'s cover $X$. Let $H\subseteq X\times|mD/\omega|$ be the universal 
family of divisors defined by the incidence relation and $H_B\rightarrow B$ the 
restriction to $B$. Take a log resolution of $H_B\subseteq X\times B$ over the 
generic point of $B$ and extend it over an open subset $U$ of $B$. By assumption the 
log resolution over the generic point of $B$ has a unique exceptional divisor of 
discrepancy $-1$, since this is true over $Q\subseteq B$. Let $Y$ be the image of this 
unique exceptional divisor in $X\times B$ with the natural projection map $\pi:Y\rightarrow X$. 
By construction $\pi:Y\rightarrow X$ dominates $X$. 

Possibly taking a finite cover of $B$ and passing to an open subset of $B$, we may 
assume that any fiber $V_t$ of $f:Y\rightarrow B$ over $t\in B$ is a non-klt center 
of $K_X+\Delta+\Delta_t$. Possibly passing to an open subset of $B$, we may assume 
that $f:Y\rightarrow B$ is flat and $B$ maps into the Hilbert scheme. Replace $B$ by 
the normalization of the closure of its image in the Hilbert scheme and $Y$ by the 
normalization of the pullback of the universal family. After possibly cutting by 
hyperplanes in $B$, we may assume that $\pi$ is generically finite and dominant. 
The resulting family is the required covering family of tigers.
\end{proof}

In fact, the original construction of covering families of tigers is carried out in 
a more general setting. For a topological space $X$, we say that a subset $P$ is 
\emph{countably dense} if $P$ is not contained in the union of countable many 
closed subsets of $X$. 
\begin{corollary}\label{gc} Let $(X,\Delta)$ be a log pair and let $D$ be a big 
$\QQ$-Cartier divisor. Let $\omega$ be a positive rational number. Let $P$ be a 
countably dense subset of $X$. If for every point $p\in P$ we may find a pair 
$(\Delta_p, V_p)$ such that $V_p$ is a pure non-klt center of $K_X+\Delta+\Delta_p$, 
where $\Delta_p\sim_\QQ D/\omega_p$ for some $\omega_p>\omega$, then we may find 
a covering family of tigers of weight $\omega$ together with a countably dense subset 
$Q$ of $P$ such that for all $q\in Q$, $V_q$ is a fiber of $\pi$.
\end{corollary}
\begin{proof} See \cite[Lemma 3.2]{Mck:fano} or \cite[Lemma 3.2]{Tod}. 
\end{proof}

As noted in Example \ref{guiding}, we can assume that the covering families of 
tigers under our consideration are always positive dimensional.
\begin{lemma}\label{dim} Let $(X,\Delta)$ be a projective klt pair and $D=-(K_X+\Delta)$ 
be a big and nef $\QQ$-Cartier divisor. A covering family of tigers $(\Delta_t,V_t)$ 
of weight $\omega>2$ is positive dimensional, i.e., $k=\dim V_t>0$.
\end{lemma}
\begin{proof} This is \cite[Lemma 3.4]{Mck:fano} and we include the proof for the 
convenience of the reader. Suppose that there exists a zero dimensional covering 
family of tigers of weight $\omega>2$. For $p_1$ and $p_2$ general, there are 
divisors $\Delta_1$ and $\Delta_2$ with $\Delta_i\sim_{\QQ}D/\omega$ such that 
$p_i$ is an isolated non-klt center of $K_X+\Delta+\Delta_i$. As $p_1$ and $p_2$ 
are general, it follows that $\Delta_2$ does not contain $p_1$ and 
${\rm Nklt}(X,\Delta+\Delta_1+\Delta_2)$ contains $p_1$  and $p_2$ as disconnected 
non-klt centers. But $-(K_X+\Delta+\Delta_1+\Delta_2)\sim (1-\frac{2}{\omega})D$ 
is nef and big if $\omega>2$. This contradicts Lemma \ref{con}.
\end{proof}

Recall that we want to cut down the dimension of non-klt centers via Theorem \ref{cut}. 
To do so, we study the associated covering families of tigers and obtain a lower bound of 
restricted volumes on the non-klt centers. If the new non-klt centers after cutting down 
the dimension are still positive dimensional, then we have to create new covering families 
of tigers associated to these new non-klt centers and repeat the process. The following 
proposition enables us to create covering families of tigers of new non-klt centers after 
cutting down the dimension. 
\begin{proposition}\label{bc} Let $(X,\Delta)$ be a log pair and let $D$ be a 
$\QQ$-Cartier divisor of the form $A+E$ where $A$ is ample and $E$ is effective. 
Let $(\Delta_t,V_t)$ be a covering family of tigers of weight $\omega$ and 
dimension $k$. Let $A_t$ be $A|_{V_t}$. If there is an open subset 
$U\subseteq B$ such that for all $t\in U$ we may find a covering family of tigers 
$(\Gamma_{t,s}, W_{t,s})$ on $V_t$ of weight $\omega'$ with respect to $A_t$, 
then for $(X,\Delta)$ we can find a covering family of tigers $(\Gamma_s,W_s)$ 
of dimension less than $k$ and weight 
    \begin{align*} \omega''=\frac{1}{1/\omega+1/\omega'}=\frac{\omega\omega'}{\omega+\omega'}. 
    \end{align*}
\end{proposition}
\begin{proof} This is \cite[Lemma 5.3]{Mck:fano}. 
\end{proof}

We will apply Proposition \ref{bc} with the ample divisor $D=-(K_X+\Delta)$. 
In the process of obtaining lower bound of the restricted volume on the 
non-klt centers, if we have one-dimensional non-klt centers, then we can control 
the restricted volume of $D$, cf. \cite[Lemma 5.3]{Mck:fano}.
\begin{corollary}\label{min} Let $(X,\Delta)$ be a log pair and let $D$ be an 
ample divisor. Let $(\Delta_t, V_t)$ be a covering family of tigers of weight 
$\omega>2$ and dimension one. Then ${\rm deg}(D|_{V_t})\leq2\omega/(\omega-2)$.
 \end{corollary}
\begin{proof} Suppose that ${\rm deg}(D|_{V_t})>2\omega/(\omega-2)$. By 
Lemma \ref{exc} and Corollary \ref{gc}, we may find a covering family 
$(\Gamma_{t,s},W_{s,t})$ of tigers of weight $\omega'>2\omega/(\omega-2)$ 
and dimension zero on $V_t$. By Proposition \ref{bc}, there exists a covering 
family of tigers of dimension zero and weight  
    \begin{align*} \omega''=\frac{\omega\omega'}{\omega+\omega'}>2, 
    \end{align*}
for $X$. This contradicts Lemma \ref{dim}. 
\end{proof} 

\section{Log Del Pezzo Surfaces}
Let $(X,\Delta)$ be an $\epsilon$-klt weak log del Pezzo surface. The minimal 
resolution $\pi:Y\rightarrow X$ of $(X,\Delta)$ is the unique proper birational 
morphism such that $Y$ is a smooth projective surface and $K_Y+\Delta_Y=\pi^*(K_X+\Delta)$ 
for some effective $\QQ$-divisor $\Delta_Y$ on $Y$. Note that minimal resolutions 
always exist for two-dimensional log pairs. It is easy to see that $(Y,\Delta_Y)$ 
is also an $\epsilon$-klt weak log del Pezzo surface with volume 
    \begin{align*} {\rm Vol}(Y,\Delta_Y)=(K_Y+\Delta_Y)^2=(K_X+\Delta_X)^2={\rm Vol}(X,\Delta_X).     
    \end{align*}
Replacing $(X,\Delta)$ by its minimal resolution, we can assume that $X$ 
is smooth. 

Write $(K_X+\Delta)^2>(2\omega)^2$. For a general point $p\in X$, let 
$\Delta_p\sim_\QQ-(K_X+\Delta)/\omega$ be an effective $\QQ$-divisor constructed from 
Lemma \ref{sd} such that $p\in{\rm Nklt}(X,\Delta+\Delta_p)$. Assume that $\omega>2$. 
By Lemma \ref{dim}, the unique minimal non-klt center $F_p$ of $(X,\Delta+\Delta_p)$ 
containing $p$ is one dimensional. Note that for general $p\in X$, $F_p\leq\Delta_p$.

\begin{lemma}\label{a'} For a very general point $p\in X$, the numerical class 
$F:=F_p$ on $X$ is well-defined and $F$ is nef. 
\end{lemma}
\begin{proof} The effective integral one cycles $F_p$ satisfy 
$F_p\leq\Delta_p\sim_\QQ -(K_X+\Delta)/\omega$ and hence form a bounded set in the 
Mori cone of curves.  As $\CC$ is uncountable, for $p\in X$ a very general point 
the numerical class $F:=F_p$ is well-defined. Since $\{F_p\}$ moves, the 
class $F$ is nef.
\end{proof}

The following lemma shows that if we assume the weight $\omega$ is large, then the 
non-klt centers $\{F_p\}$ on $X$ already possess a nearly fiber bundle structure 
analogous to a covering family of tigers.
\begin{lemma}\label{a} Assume that $\omega>3$, then $F^2=0$, i.e. $F_p\cap F_q=\emptyset$ 
for $p, q\in X$ two very general points. 
\end{lemma}
\begin{proof} Assume that $F_p\cap F_q\neq\emptyset$ for $p,q\in X$ two very general points. 
We can assume that $p\notin\Delta_q$ as $p\in X$ is very general. Since by Lemma \ref{a} 
the curve class $F=F_p$ is nef, for $H=-(K_X+\Delta)/\omega$ we have 
    \begin{align*} 1\leq F_p.F_q=F_p.F\leq \Delta_p.F={\rm deg}(H|_{F_p}),   
    \end{align*}
where the first inequality is true since $X$ is smooth. Since $H$ is big and nef, 
we can cut down the dimension of the non-klt centers by Theorem \ref{cut}\footnote{By 
adding a small multiple of $-(K_X+\Delta)$, we may assume that the inequality ${\rm deg}(H|_{F_q})\geq1$ 
is strict with a smaller modified $\omega$ and hence Theorem \ref{cut} applies.}. 

To be precise, pick $0<\delta_1\leq1$ such that the pair $(X,\Delta+\delta_1\Delta_p)$ 
is log canonical but not klt at $p$. If $(X,\Delta+\delta_1\Delta_p)=\{p\}$, then this 
contradicts the Connected Lemma \ref{con} as $p\notin\Delta_q$ and the non-klt locus ${\rm Nklt}(X,\Delta+\delta_1\Delta_p+\Delta_q)$ 
containing $p$ and $F_q$ is disconnected, while the divisor $-(K_X+\Delta+\delta_1\Delta_p+\Delta_q)$ 
is nef and big.  Hence we may assume that ${\rm Nklt}(X,\Delta+\delta_1\Delta_p)$ is one dimensional 
in a neighborhood of $p$. In particular, $F_p\subseteq{\rm Nklt}(X,\Delta+\delta_1\Delta_p)$ is 
the minimal non-klt center containing $p$. By Theorem \ref{cut}, there exists rational 
numbers $0<\delta\ll1$, $0<c<1$, and an effective $\QQ$-divisor $B_p\equiv H$ such that 
\mbox{${\rm Nklt}(X,\Delta+(1-\delta)\delta_1\Delta_p+cB_p)=\{p\}$} in a neighborhood of 
$p$. It follows that the set of non-klt centers ${\rm Nklt}(X,\Delta+(1-\delta)\delta_1\Delta_p+cB_p+\Delta_q)$ 
containing $p$ and $F_q$ is disconnected but the divisor $-(K_X+\Delta+(1-\delta)\delta_1\Delta_p+cB_p+\Delta_q)$ 
is nef and big as $\omega>3$. This again contradicts the Connected Lemma \ref{con}. 
\end{proof}

\begin{theorem}\label{I} Let $(X,\Delta)$ be an $\epsilon$-klt weak log del Pezzo 
surface. Then the anticanonical volume ${\rm Vol}((-K_X+\Delta))=(K_X+\Delta)^2$ satisfies
    \begin{align*} (K_X+\Delta)^2\leq\max\{64,\frac{8}{\epsilon}+4\}. 
    \end{align*}
\end{theorem}
\begin{proof} Replacing $(X,\Delta)$ by its minimal resolution, we may assume 
that $X$ is smooth. Write $(K_X+\Delta)^2>(2\omega)^2$. For each general point $p\in X$, 
by Lemma \ref{sd}, there exists an effective \mbox{$\QQ$-divisor} $\Delta_p\sim_\QQ-(K_X+\Delta)/\omega$ 
such that $p\in{\rm Nklt}(X,\Delta+\Delta_p)$. From Lemma \ref{dim}, we may assume that 
$\omega>2$ and the unique minimal non-klt center $F_p\subseteq{\rm Nklt}(X,\Delta+\Delta_p)$ 
containing $p$ is one dimensional. Note that $F_p\leq\Delta_p$ for general $p\in X$. By Lemma 
\ref{a'} and \ref{a}, we may assume that $\omega>3$ and for very general $p\in X$ the numerical 
class $F$ of $F_p$ is well-defined and nef with $F^2=0$.

For two very general points $p,q\in X$, $\Delta_p.\Delta_q>0$ and hence 
$F_p={\rm Supp}(F_p)\subsetneqq{\rm Supp}(\Delta_p)$: Otherwise $\Delta_q\equiv\Delta_p\leq NF_p$ 
for some $N>0$ and $0<\Delta_p.\Delta_q\leq N^2F_p^2=N^2F^2=0$, a contradiction. 
By the Connectedness Lemma \ref{con}, ${\rm Nklt}(X,\Delta+\Delta_p+\Delta_q)\supseteq{F_p}\cup{F_q}$ 
is connected. Denote \mbox{$E_p={\rm Supp}(\Delta_p)-F_p\neq0$}. By Lemma \ref{a}, $F_p\cap F_q=\emptyset$ 
and hence $E_p$ must contain a connected curve $E\leq E_p$ such that $F_p.E\neq0$, $F_q.E\neq0$, 
and the set ${\rm Nklt}(X,\Delta+\Delta_p+\Delta_q)\supseteq{F_p}\cup{F_q}\cup{E}$. 
Furthermore, we can assume that $E$ is irreducible since $E.F_q\neq0$ as $F_q\equiv F_p$ 
for $q\in X$ a very general point. %By symmetry and the $\epsilon$-klt condition, $E$ satisfies $\frac{\epsilon}{2}E\leq\Delta_p$ (cf. Example \ref{guiding}).

Suppose that $E^2\geq0$ and hence $E$ is nef. Since ${\rm Nklt}(X,\Delta+\Delta_p+\Delta_q)\supseteq{F_p}\cup{F_q}\cup{E}$, 
we have $\Delta+\Delta_p+\Delta_q\geq E$ and $(\Delta+\Delta_p+\Delta_q-E).E\geq0$. 
For $H=-(K_X+\Delta)/\omega$, we see that  
    \begin{align*} 2\geq 2-2g_a(E)\geq&-(K_X+E).E-(\Delta+\Delta_p+\Delta_q-E).E \\
				     =&-(K_X+\Delta+\Delta_p+\Delta_q).E \\
				     =&(\omega-2)H.E.
    \end{align*}
Write $\Delta_p=\Delta'_p+\alpha E$ where $\Delta_p'\wedge E=0$, $\Delta'_p\geq F_p$, and $\alpha >0$, 
we have
    \begin{align*} H.E=\Delta_p.E=(\Delta'_p+\alpha E).E\geq F_p.E\geq1.     
    \end{align*}
The last inequality follows from the fact that $X$ is smooth and $F_p.E>0$. Combine the 
two inequalities above, we obtain $\omega\leq4$.

Hence we may assume that $E^2<0$, and thus 
    \begin{align*} -2&\leq2g_a(E)-2=(K_X+E).E \\
                                    &=(K_X+\Delta).E+(1-\epsilon-a_E)E^2-\Delta'.E+\epsilon E^2 \leq\epsilon E^2,    
    \end{align*}
where $\Delta=\Delta'+a_E E$ with $\Delta'\wedge E=0$ and $a_E\in[0,1-\epsilon)$ 
by the $\epsilon$-klt condition. This implies that $1\leq-E^2\leq2/\epsilon$, where 
the first inequality follows from the fact that $E^2\in\ZZ$ as $X$ is smooth. 
Since $F^2=0$ for $F$ the numerical class of $F_p$ where $p\in X$ is very general, by 
Nakai's criterion the divisor $H_s=F+sE$ with $0<s\leq1/(-E^2)$ is nef and big. By 
the Hodge index theorem (see \cite[V 1.1.9(a)]{H:ag}), we get the inequality 
    \begin{equation}\label{3.0} (K_X+\Delta)^2\leq\frac{(-(K_X+\Delta).H_s)^2}{H_s^2}.
    \end{equation}
From $\Delta.F\geq0$ and $F^2=0$, we have that
    \begin{equation}\label{3.1} -(K_X+\Delta).F\leq-(K_X+F).F\leq2.      
    \end{equation}
Also for $\Delta=\Delta'+a_EE$ with $\Delta'\wedge E=0$ and $a_E\in[0,1-\epsilon)$, 
we have that 
    \begin{align*}\label{3.2} -(K_X+\Delta).E=&-K_X.E-\Delta'.E-a_E E^2 \\ 
                                         \leq &E^2+2-a_E E^2=(a_E-1)(-E^2)+2\leq2-\epsilon(-E^2).  \tag{4.3}
    \end{align*}
Put $s=1/(-E^2)$, all together we get 
    \begin{align*} (K_X+\Delta)^2&\leq\frac{(-(K_X+\Delta).(F+sE))^2}{H_s^2} \\
                                 &\leq\frac{(2+s(2-\epsilon(-E^2)))^2}{2sE.F+s^2E^2} \\
                                 &\leq(-E^2)(2-\epsilon+\frac{2}{-E^2})^2 \\
                                 &=(-E^2)(2-\epsilon)^2+4(2-\epsilon)+\frac{4}{-E^2} \\
                                 &\leq\frac{2}{\epsilon}(2-\epsilon)^2+4(2-\epsilon)+4 \\
                                 &=\frac{8}{\epsilon}+4-2\epsilon
    \end{align*}  
where the first inequality is \eqref{3.0}, the second inequality follows from 
\eqref{3.1}, \eqref{3.2}, and $F^2=0$, the third inequality is given by ignoring 
the term $sE.F\geq0$, and the last inequality uses \mbox{$1\leq-E^2\leq2/\epsilon$}. 
\end{proof}

\begin{remark} Note that by applying Corollary \ref{min} one can only obtain 
an upper bound of order $1/\epsilon^2$. Hence Theorem \ref{I} is a non-trivial 
result.   
\end{remark}

\section{Log Fano Threefolds of Picard Number One}
Let $(X,\Delta)$ be an $\epsilon$-klt $\QQ$-factorial log $\QQ$-Fano threefold 
of Picard number $\rho(X)=1$. Note that by hypothesis $X$ is $\epsilon$-klt and 
$-K_X$ is ample with $-K_X^3\geq{\rm Vol}(-(K_X+\Delta))=-(K_X+\Delta)^3$. Hence it 
is sufficient to assume that $X$ is an $\epsilon$-klt $\QQ$-factorial $\QQ$-Fano 
threefold of Picard number $\rho(X)=1$  and to find an upper bound of ${\rm Vol}(-K_X)=-K_X^3$. 
We will obtain an upper bound of the anticanonical volumes by studying covering 
families of tigers. The weight of any covering families of tigers in our study 
will always be the weight with respect to $-K_X$.

Let $X$ be an $\epsilon$-klt $\QQ$-factorial $\QQ$-Fano threefold of Picard 
number $\rho(X)=1$ and write the anticanonical volume ${\rm Vol}(-K_X)=-K_X^3>(3\omega)^3$ for some positive 
rational number $\omega$. Denote $D=-2K_X$, we have $D^3>(6\omega)^3$. By Lemma 
\ref{sd}, we can fix an affine open subset $U\subseteq X$ such that for each 
$p\in U$ there exists an effective divisor $\Delta_p\sim_\QQ D/\omega$ with 
${\rm mult}_p\Delta_p\geq6$. 
%We do not know whether $\Delta_p$ creates an exceptional non-klt center or not as in Lemma \ref{exc}. If it does, then we lose control on multiplicity. 
We pick divisors $\Delta_p$'s in the following systematic way so that we can 
control their multiplicities uniformly.

\subsection{Construction}\label{const}
Let $\Delta_U\subseteq U\times U$ be the diagonal and $\mathcal{I}_\mathcal{Z}$ 
be the ideal sheaf of the subvariety $\mathcal{Z}=\overline{\Delta_U}\subseteq X\times U$. 
For each $p\in U$, by the existence of $\QQ$-divisor $\Delta_p\sim_\QQ D/\omega$ 
with ${\rm mult}_p\Delta_p\geq6$, there exists $m_p>0$ such that $L_{m_p}=m_pD/\omega$ 
is Cartier and $H^0(X,L_{m_p}\otimes\mathcal{I}_p^{\otimes 6m_p})\neq0$. 
In particular, we can write $U=\cup U_m$ where $m>0$ runs through all sufficiently 
divisible integers such that $L_m=mD/\omega$ is Cartier and $U_m=\{p\in U|H^0(X,L_m\otimes\mathcal{I}_p^{\otimes 6m})\neq0\}$. 
Moreover, each $U_m$ is locally closed in $X$ by \cite[III, Theorem 12.8]{H:ag} and 
$X=\cup\overline{U_m}$. Since the base field $\CC$ is uncountable, $X$ can not be 
a countable union of locally closed subsets. Thus there exists some $m>0$ such 
that $U_m$ is dense in $X$. 

Fix an $m>0$ such that $L_m=mD/\omega$ is Cartier and $U_m=\{p\in U|H^0(X,L_m\otimes\mathcal{I}_p^{\otimes 6m})\neq0\}$ 
is dense in $X$. Denote ${\rm pr}_X:X\times U\rightarrow X$ and ${\rm pr}_U:X\times U\rightarrow U$ 
the projection maps. Since ${\rm pr}_U:X\times U\rightarrow U$ is flat, by 
\cite[III,Theorem 12.11]{H:ag}, after restricting to a smaller open affine subset 
of $U$, we can assume that the map 
\begin{align*} ({\rm pr}_U)_*({\rm pr}_X^*L_m\otimes\mathcal{I}_\mathcal{Z}^{\otimes 6m})\otimes\CC(p)\rightarrow 
		H^0(X,L_m\otimes\mathcal{I}_p^{\otimes 6m}),
\end{align*}
is an isomorphism for each $p\in U$ where $\mathcal{I}_p$ is the ideal sheaf of $p\in U$. 
Since $U_m$ is dense in $U$, the sheaf $({\rm pr}_U)_*({\rm pr}_X^*L_m\otimes\mathcal{I}_\mathcal{Z}^{\otimes 6m})\neq0$ 
on $U$ and hence $H^0(X\otimes U,{\rm pr}_X^*L\otimes\mathcal{I}_\mathcal{Z}^{\otimes 6m})\neq0$ 
as $U$ is affine. Let $s\in H^0(X\otimes U,{\rm pr}_X^*L\otimes\mathcal{I}_\mathcal{Z}^{\otimes 6m})$ 
be a nonzero section with $F={\rm div}(s)$ the corresponding divisor on $X\times U$. 
For each $p\in U$, denote $F_p=F\cap(X\times\{p\})$ the associated divisor on 
$X\cong X\times\{p\}$. Since ${\rm mult}_\mathcal{Z}(F)\geq6m$, by Lemma \ref{fiber} 
below, the $\QQ$-divisor $\Deltaʼ_p=F_p/m\sim_\QQ D/\omega$ on $X$ satisfies 
${\rm mult}_p\Delta_p\geq6$ for general $p\in U$. 

\begin{lemma}{$($\cite[Lemma 5.2.11]{Laz:p1}$)$}\label{fiber} Let $g:M\rightarrow T$ 
be a morphism of smooth varieties, and suppose that $\mathcal{Z}\subseteq M$ is an irreducible 
subvariety dominating $T$:
\begin{center}
\begin{tikzpicture}
\node (A) at (-2,0) {$\mathcal{Z}$};
\node (B) at (0,-2){ $T.$};
\node (C) at (2,0){$M$};
\path[->] (A) edge node[below]{}(B);
\path[->] (C) edge node[right]{$g$}(B);
\path[right hook->] (A) edge node[right]{}(C);
\end{tikzpicture}
\end{center}
Let $F\subseteq M$ be an effective divisor. For a general point 
$t\in T$ and an irreducible component $\mathcal{Z}_t'\subseteq \mathcal{Z}_t$,
${\rm mult}_{\mathcal{Z}'_t}(M_t,F_t)={\rm mult}_{\mathcal{Z}}(M,F)$, 
where ${\rm mult}_\mathcal{Z}(M,F)$ is the multiplicity of the divisor $F$ 
on $M$ along a general point of the irreducible subvariety $\mathcal{Z}\subseteq M$ 
and similarly for ${\rm mult}_{\mathcal{Z}'_t}(M_t,F_t)$.
\end{lemma}

For a given collection of $\QQ$-divisors $\{\Deltaʼ_p=F_p/m\sim_\QQ D/\omega|p\in U\ {\rm general}\}$ 
associated to a nonzero section in $H^0(X\otimes U,{\rm pr}_X^*L\otimes\mathcal{I}_\mathcal{Z}^{\otimes 6m})$ 
as above, by Lemma \ref{exc}, we can modify the $\Delta_p$'s so that the unique 
non-klt centers $V_p\subseteq{\rm Nklt}(X,\Delta_p)$ passing through $p$ are 
exceptional. By Lemma \ref{tiger} (or in general Corollary \ref{gc}), we can construct 
covering families of tigers from these divisors. 

In order to obtain an upper bound of $\omega$, which is sufficient for bounding the anticanonical volumes, 
we will pick up a ``well-behaved'' nonzero section $s\in H^0(X\otimes U,p^*L\otimes\mathcal{I}_\mathcal{Z}^{\otimes 6m})$ 
and study the corresponding covering families of tigers.

\subsection{Cases} By Section \ref{const}, there exists an open affine subset $U\subseteq X$ and 
an integer $m>0$ such that $H^0(X\otimes U,{\rm pr}_X^*L\otimes\mathcal{I}_\mathcal{Z}^{\otimes 6m})\neq0$. 
Let $s\in H^0(X\times U,{\rm pr}_X^*L\times\mathcal{I}_\mathcal{Z}^{\otimes 6m})$ be a 
nonzero section with divisor $F={\rm div}(s)$ on $X\times U$ and 
$\{\Deltaʼ_p=F_p/m\sim_\QQ D/\omega|p\in U\}$ be the associated collection 
of $\QQ$-divisors. We consider two cases: 
\begin{enumerate}
	\item (\textbf{Small multiplicity}) For each irreducible component $\mathcal{W}$ 
              of ${\rm Supp}(F)$ passing through $\mathcal{Z}$, ${\rm mult}_{\mathcal{W}}(F)\leq3m$, 
	      i.e., for general $p\in U$ we have ${\rm mult}_{W}(\Delta_p)\leq3$ for 
	      any irreducible component $W$ of ${\rm Supp}(\Delta_p)$ passing through $p$. 
	      After differentiating $F$, we will construct a ``well-behaved'' 
              covering family of tigers of dimension one. We will derive an upper bound 
	      of $\omega$ by studying this covering family of tigers. See Section \ref{IV.1}.
        \item (\textbf{Big multiplicity}) There exists an irreducible component $\mathcal{W}$ 
              of ${\rm Supp}(F)$ passing through $\mathcal{Z}$ with multiplicity ${\rm mult}_{\mathcal{W}}(F)>3m$, 
	      i.e., for general $p\in U$ we have ${\rm mult}_{W}(\Delta_p)>3$ for 
	      some irreducible component $W$ of ${\rm Supp}(\Delta_p)$ passing through $p$.
              We will construct a covering family of tigers of dimension two and derive 
              an upper bound of $\omega$ by studying the geometry of this covering family 
              of tigers. See Section \ref{IV.2}.
\end{enumerate}

To pick a ``well-behaved'' nonzero section in $H^0(X\otimes U,{\rm pr}_X^*L\otimes\mathcal{I}_\mathcal{Z}^{\otimes 6m})$, 
we will apply the following proposition.

\begin{proposition}{$($\cite[Proposition 5.2.13]{Laz:p1}$)$}\label{diff} Let 
$X$ and $U$ be smooth irreducible varieties, with $U$ affine, and suppose that 
$\mathcal{Z}\subseteq \mathcal{W}\subseteq X\times U$ are irreducible subvarieties 
such that $\mathcal{W}$ dominates $X$. Fix a line bundle $L$ on $X$, and suppose we are 
given a divisor $F\in|{\rm pr}_X^*(L)|$ on $X\times U$. Write $l={\rm mult}_\mathcal{Z}(F)$ and $k={\rm mult}_\mathcal{W}(F)$.    
After differentiating in the parameter directions, there exists a divisor 
$F'\in|{\rm pr}_X^*(L)|$ on $X\times U$ with the property that ${\rm mult}_\mathcal{Z}(F')\geq l-k, \ {\rm and}\ \mathcal{W}\nsubseteq{\rm Supp}(F')$.    
\end{proposition}

\subsection{Small multiplicity}\label{IV.1}
Let $X$ be an $\epsilon$-klt $\QQ$-Fano threefold of Picard number one 
and write ${\rm Vol}(-K_X)=-K_X^3>(3\omega)^3$ for some positive rational 
number $\omega$. Denote $D=-2K_X$, we have $D^3>(6\omega)^3$. By Section \ref{const}, 
there is an integer $m>0$ such that $L=mD/\omega$ is Cartier and an open affine 
subset $U\subseteq X$ such that $H^0(X\times U,{\rm pr}_X^*L\otimes\mathcal{I}_\mathcal{Z}^{\otimes 6m})\neq0$. 
We fix a nonzero section $s\in H^0(X\times U,{\rm pr}_X^*L\otimes\mathcal{I}_\mathcal{Z}^{\otimes 6m})$ 
with $F={\rm div}(s)$ on $X\times U$. 

\begin{proposition}\label{one} With the set up above. Assume that $\omega>4$. 
If we are in the case where all the irreducible components $\mathcal{W}$ 
of ${\rm Supp}(F)$ passing through $\mathcal{Z}$ satisfy ${\rm mult}_{\mathcal{W}}(F)\leq3m$, 
then $\omega<8/\epsilon+4$. In particular, there is an upper bound for the volume 
    \begin{align*} {\rm Vol}(-K_X)=-K_X^3\leq (\frac{24}{\epsilon}+12)^3.    
    \end{align*} 
\end{proposition}
\begin{proof} Let $M$ be the maximum of ${\rm mult}_{\mathcal{W}}(F)$ among all 
the irreducible components $\mathcal{W}$ of ${\rm Supp}(F)$ passing through $\mathcal{Z}$. 
Then $M\leq3m$ by the hypothesis. For a fixed irreducible component $\mathcal{W}$ of ${\rm Supp}(F)$ 
passing through $\mathcal{Z}$, we can apply Proposition \ref{diff} to $F$. 
We obtain a divisor $F'\in|{\rm pr}_X^*(L)\otimes\mathcal{I}_\mathcal{Z}^{\otimes 6m-M}|$ 
with the property that     
    \begin{align*} {\rm mult}_\mathcal{Z}(F')\geq(6m-M)\geq3m,\ {\rm and}\ 
		    \mathcal{W}\nsubseteq{\rm Supp}(F').
    \end{align*}
Since there are only finitely many irreducible components of ${\rm Supp}(F)$ passing 
through $\mathcal{Z}$, by taking a generic differentiation, it follows that for 
a general divisor $F''\in|{\rm pr}_X^*(L)\otimes\mathcal{I}_\mathcal{Z}^{\otimes 6m-M}|$ 
we have $\mathcal{W}\nsubseteq{\rm Supp}(F'')$ for any irreducible component $\mathcal{W}$ 
of ${\rm Supp}(F)$ passing through $\mathcal{Z}$. In particular, the base locus 
${\rm Bs}(|{\rm pr}_X^*L\otimes\mathcal{I}_\mathcal{Z}^{\otimes 6m-M}|)$ 
contains no codimension one components in a neighborhood of $\mathcal{Z}$. 

Let $G$ be a general divisor in $|{\rm pr}_X^*L\otimes\mathcal{I}_\mathcal{Z}^{\otimes 6m-M}|$ 
and $\Delta_p=G_p/m$ for $p\in U$ general the corresponding $\QQ$-divisors on $X$. 
It follows that $p\in{\rm Nklt}(K_X+\Delta_p)$ as ${\rm mult}_p\Delta_p\geq3$. 
The minimal non-klt center $V_p\subseteq{\rm Nklt}(K_X+\Delta_p)$ 
passing through $p$ must be positive dimensional by Lemma \ref{dim} as the weight of 
$\Delta_p$ is $\omega/2>2$. Note that we may replace $|{\rm pr}_X^*L\otimes\mathcal{I}_\mathcal{Z}^{\otimes 6m-M}|$ 
by $|{\rm pr}_X^*L^{\otimes k}\otimes\mathcal{I}_\mathcal{Z}^{\otimes k(6m-M)}|$ for any $k\geq1$ 
and hence we may assume that $m\gg0$. In particular, we have $0\leq{\rm mult}_{W}\Delta_p\ll1$ for $W$ any irreducible component 
of ${\rm Supp}(\Delta_p)$, and $V_p$ can only be one-dimensional.

Let $\pi:Y\rightarrow X$ and $f:Y\rightarrow B$ be a one dimensional covering family 
of tigers of weight $\omega'\geq\omega/2$ constructed from the $\Delta_p$'s above 
by Lemma \ref{exc} and Lemma \ref{tiger}. By abuse of notation, we still denote $\Delta_p$'s 
the divisors associated to this covering family of tigers. 

Choose $p, q\in U\subseteq X$ general. By Lemma \ref{con}, the non-klt locus 
${\rm Nklt}(\pi^*(K_X+\Delta_p+\Delta_q))\supseteq V_p\cup V_q$ 
on $Y$ is connected and it contains a one dimensional cycle $C_{p,q}$ connecting 
$V_p$ and $V_q$. Since $Y$ is normal, an irreducible component $C$ of $C_{p,q}$ 
intersecting $V_q$ satisfies $C\cap Y_{\rm sm}\neq\emptyset$ for $p, q\in X$ general. 
Since $C$ is in ${\rm Nklt}(\pi^*(K_X+\Delta_p+\Delta_q))$, by symmetry we 
have ${\rm mult}_C(\pi^*(\Delta_p))>\epsilon/2$. 

Suppose that $\Sigma\subseteq{\rm Supp}(\pi^*(\Delta_p))$ is an irreducible component containing 
$C$. If $f(\Sigma)=f(C)$ is a curve, then $V_p\subseteq\Sigma=f^{-1}(f(C))$ as the general fiber of $f:Y\rightarrow B$ 
is irreducible. Moreover, we can assume that $\Sigma$ is not $\pi$-exceptional as there are only 
finitely many $\pi$-exceptional divisors and we choose $p\in X$, and hence $V_p$, general. 
Note that there can only be one such $\Sigma$ once we fix $p\in X$ and $C$. In particular, $\Sigma\subseteq{\rm Supp}(\pi^{-1}_*(\Delta_p))$ 
is an irreducible component containing $V_p$, and we can write $\pi^*(\Delta_p)=\Delta'+\lambda\Sigma$ with 
$\Delta'\wedge\Sigma=0$. Moreover, $\lambda\leq1/m$, where $m\gg0$ by our choice of $\Delta_p$ 
with $0\leq{\rm mult}_{W}\Delta_p\ll1$ for $W$ any irreducible component of ${\rm Supp}(\Delta_p)$.
Also, ${\rm mult}_C\Sigma=1$ since $\Sigma$ is smooth along $C$ as $f(C)$ passes through a general 
point of $B$ and $Y$ is smooth in codimension one. 

Choose a general point $b'\in f(C)$, we have that $Y_{b'}$ is a general fiber of $f:Y\rightarrow B$ and  
    \begin{align*} \frac{2}{\frac{\omega}{2}-2}\geq\frac{2}{\omega}(-K_X.V_t)=\pi^*(\Delta_p).Y_{b'}=(\Delta'+\lambda\Sigma).Y_{bʼ}>\frac{\epsilon}{2}-\frac{1}{m},
    \end{align*}   
where the first inequality follows from Corollary \ref{min}. The second inequality follows from 
$\Sigma.Y_{bʼ}\geq0$ and ${\rm mult}_C\Delta'={\rm mult}_C(\pi^*(\Delta_p))-\lambda{\rm mult}_C\Sigma$. 
Since $m\gg0$, we get $\omega\leq8/\epsilon+4$. 
\end{proof}

\begin{remark} In the proof of Proposition \ref{one}, the difficulty arises because 
in general the one cycle $C$ might be contained in ${\rm Supp}(\pi^{-1}_*(\Delta_p))$. 
In this case, one can not see the contribution of the $\epsilon$-klt condition from the 
intersection number $\pi^*\Delta_p.Y_b$ for $Y_b$ a general fiber over $f(C)\subseteq B$ 
as $Y_b\subseteq{\rm Supp}(\pi^{-1}_*(\Delta_p))$, cf., Example 
\ref{guiding}. The differentiation argument eliminates the contribution of irreducible 
components of ${\rm Supp}(\pi^{-1}_*(\Delta_p))$ along $Y_b$.  
\end{remark}

\subsection{Big multiplicity}\label{IV.2}
Again, let $X$ be an $\epsilon$-klt $\QQ$-factorial $\QQ$-Fano threefold of Picard 
number one. Write ${\rm Vol}(-K_X)=-K_X^3>(3\omega)^3$ for some positive rational 
number $\omega$ and denote \mbox{$D=-2K_X$}. As before, by Section \ref{const}, there is an 
integer $m>0$ such that $L=mD/\omega$ is Cartier and an open affine subset $U\subseteq X$ 
such that $H^0(X\times U,{\rm pr}_X^*L\otimes\mathcal{I}_\mathcal{Z}^{\otimes 6m})\neq0$. 
We fix a nonzero section $s\in H^0(X\times U,{\rm pr}_X^*L\otimes\mathcal{I}_\mathcal{Z}^{\otimes 6m})$ 
with $F={\rm div}(s)$ on $X\times U$. We now consider the case where there exists 
an irreducible component $\mathcal{W}$ of ${\rm Supp}(F)$ passing through $\mathcal{Z}$ 
with multiplicity ${\rm mult}_{\mathcal{W}}(F)>3m$.

\begin{lemma}\label{big1} If there exists an irreducible component 
$\mathcal{W}$ of ${\rm Supp}(F)$ passing through $\mathcal{Z}$ with multiplicity 
${\rm mult}_{\mathcal{W}}(F)>3m$, then there exists a covering family of tigers 
of dimension two and weight $\omega'\geq\omega/2$. 
\end{lemma}
\begin{proof} Fix $\mathcal{W}$ to be one of these irreducible components of ${\rm Supp}(F)$. We have 
the inclusions $\mathcal{Z}\subseteq\mathcal{W}\subseteq X\times U$ with the projection map $\mathcal{W}\rightarrow U$.
Cutting down by hyperplanes on $U$ and restricting to a smaller open subset of $U$, 
we may assume that $\mathcal{W}\rightarrow U$ factors through a Hilbert 
scheme of $X$ and $\mathcal{W}\rightarrow X$ is generically finite. Replace $U$ 
by the normalization of the closure of its image in the Hilbert scheme and 
$\mathcal{W}$ by the normalization of universal family. We obtain maps $\pi:Y\rightarrow X$ 
and $f:Y\rightarrow B$. Note that a general fiber $Y_b$ is two dimensional. 
We claim that the pairs $(\Delta_b=\pi_*(Y_b),V_b=Y_b)$ is a two dimensional 
covering of tigers of weight $\omega'\geq\omega/2$.

Since $X$ is $\QQ$-factorial and $\rho(X)=1$, the integral divisor $\Delta_b=\pi_*(Y_b)$ 
for any $p\in B$ on $X$ is $\QQ$-linear equivalent to a multiple of $-K_X$. 
Since $\mathcal{W}\leq F$, we have $\pi_*(Y_b)\leq F_b$ for general $b\in B$. 
In particular, $\pi_*(Y_b)\sim_\QQ-K_X/\omega'$ for some 
$\omega'\geq\omega/2$. Since any two general divisors $\pi_*(Y_{b_i})$, $i=1,2$, 
on $X$ are $\QQ$-linear equivalent as the base field is uncountable, and it is 
clear that $V_t=\pi(Y_b)$ is the minimal non-klt center of ${\rm Nklt}(X,\Delta_b)$, 
and the lemma follows.
\end{proof}

Let $\pi:Y\rightarrow X$ with $f:Y\rightarrow B$ be a covering family of tigers 
of dimension two and weight $\omega'\geq\omega/2$ given by Lemma \ref{big1}. We first 
deal with case where $\pi:Y\rightarrow X$ is not birational. 

\begin{proposition} Suppose that the two dimensional covering family of tigers 
$\pi:Y\rightarrow X$ with $f:Y\rightarrow B$ of weight $\omega'\geq\omega/2$ 
is not birational and assume that $\omega>12$, then $\omega\leq24/\epsilon+12$. 
In particular, there is an upper bound of volume
    \begin{align*} {\rm Vol}(-K_X)=-K_X^3\leq (\frac{72}{\epsilon}+36)^3.    
    \end{align*} 
\end{proposition}
\begin{proof} Let $d\geq2$ be the degree of $\pi:Y\rightarrow X$. Fix an open 
subset $U\subseteq X$ such that for a general point $p\in U$ there are $d$ divisors 
$\Delta_p^{t_i}$, for some $t_1,...,t_d\in B$, with $\pi(Y_{t_i})\subseteq{\rm Nklt}(X,\Delta_p^{t_i})$ 
the unique minimal non-klt center passing through $p$. Consider the collection 
of $\QQ$-divisors \mbox{$\{\Delta_p'=\frac{6}{d}\sum_{i=1}^d\Delta_p^{t_i}|p\in U\}$}, 
then ${\rm mult}_p\Delta'_p\geq 6$, ${\rm mult}_{W'}\Delta'_p=\frac{6}{d}\leq3$ for 
$W'\subseteq{\rm Supp}(\Delta_p')$ any irreducible component, and $\Delta_p'\sim_\QQ\frac{-K_X}{d\omega'/6}$. 

By the same construction as in Section \ref{const}, possibly after shrinking $U$ to a 
smaller open affine subset, there exists an integer $m>0$ such that 
$H^0(X\times U,{\rm pr}_X^*L\otimes\mathcal{I}_\mathcal{Z}^{\otimes 6m})\neq0$ 
where $L=6m(-K_X)/d\omega'$ is Cartier. Let $t\in H^0(X\times U,{\rm pr}_X^*L\otimes\mathcal{I}_\mathcal{Z}^{\otimes 6m})$ 
be a general nonzero section and $G={\rm div}(t)$ be the associated divisor on $X\times U$. 
Note that ${\rm mult}_{\mathcal{Z}}(G)\geq 6m$ and ${\rm mult}_{\mathcal{W}}(G)\leq6m/d\leq3m$ 
for any irreducible component $\mathcal{W}$ of ${\rm Supp}(G)$ passing through $\mathcal{Z}$. 
Indeed, we know that for general $p\in U$ there is the divisor $\Delta_p'$ with ${\rm mult}_p\Delta'_p\geq 6$ 
and ${\rm mult}_{W'}\Delta'_p=\frac{6}{d}\leq3$ for any irreducible component $W'\subseteq{\rm Supp}(\Delta_p')$. 
Since $t$ is a general section, $t_p=t|_{X\times\{p\}}$ is also a general section for general $p\in U$. 
Using Lemma \ref{fiber} to compute the multiplicity, we obtain 
${\rm mult}_{\mathcal{W}}(G)={\rm mult}_{\mathcal{W}_p}(G_p)\leq m\cdot{\rm mult}_{W'}\Delta'_p\leq3m$, 
where $G_p={\rm div}(t_p)$ and ${\mathcal{W}_p}$ is any irreducible component of ${\rm Supp}(G_p)$. 

By a differentiation argument and the same construction as in Proposition \ref{one}, 
there is a covering family of tigers $(\Delta_t,V_t)$ of dimension one and weight 
$\omega''\geq d\omega'/6\geq d\omega/12$, which satisfies the property that the base 
locus ${\rm Bs}(|{\rm pr}_X^*L\otimes\mathcal{I}_\mathcal{Z}^{\otimes 6m-M}|)$ 
contains no codimension one components in a neighborhood of $\mathcal{Z}$, where 
$M$ is the maximum of ${\rm mult}_{\mathcal{W}}(G)$ amongst all the irreducible 
components $\mathcal{W}$ of ${\rm Supp}(G)$ passing through $\mathcal{Z}$. 
Hence by Corollary \ref{min}, we get 
    \begin{align*} \frac{2}{\omega''-2}\geq\frac{1}{\omega''}(-K_X.V_t)=\pi^*\Delta_p.Y_b\geq\frac{\epsilon}{2}.
    \end{align*}   
In particular, 
    \begin{align*} \frac{4}{\epsilon}+2\geq\omega''\geq\frac{d\omega}{12}\geq\frac{\omega}{6},
    \end{align*}
and $\omega\leq24/\epsilon+12$.
\end{proof}

\theoremstyle{ass1}
\newtheorem*{ass1}{Assumption}
\begin{ass1} From now on, we assume that $\pi:Y\rightarrow X$ with $f:Y\rightarrow B$ 
is a \textbf{birational} covering family of tigers of dimension two and weight $\omega'\geq\omega/2$. 
Write $K_Y+\Gamma-R=\pi^*K_X$ where $\Gamma$ and $R$ are effective divisors on $Y$ 
with no common components.
\end{ass1}

\begin{lemma}\label{Exc} There is a $\pi$-exceptional divisor $E$ on $Y$ 
dominating $B$. In particular, $\pi:Y\rightarrow X$ is not small.
\end{lemma}
\begin{proof} Suppose that there is no $\pi$-exceptional divisors dominating $B$. 
Let $A_B$ be a sufficiently ample divisor on $B$ and $A_Y=f^*A_B$ the pull-back. 
Since $\rho(X)=1$, the divisor $A_X=\pi_*A_Y$ on $X$ is ample and $\pi^*A_X=A_Y+G$ 
for some effective $\pi$-exceptional divisor $G$. By assumption $f(G)\subseteq B$ 
has codimension one and hence $A_Y+G\leq f^*H$ for some divisor $H$ on $B$. This 
is a contradiction since then $A_Y+G$ is not big but $\pi^*A_X$ is.
\end{proof}

The following lemma is crucial for computing the restricted volume. The key point is 
that it allows us to control the negative part of the subadjunction $-K_X|_{V_t}$. 
Note that the proof fails in higher dimensions, cf. \cite[Lemma 6.2]{Mck:fano}.
\begin{lemma}\label{two} Let $E$ be a $\pi$-exceptional divisor dominating $B$. 
For general points $p,q\in X$ we have that $E\subseteq{\rm Nklt}(K_Y+\Gamma-R+\pi^*(\Delta_p+\Delta_q)).$ 
In particular, denote $H=\pi^*(-K_X)$. For any $\pi$-exceptional divisor $E$ 
dominating $B$ we have  
    \begin{align*} \frac{2}{\omega'}H\sim_\QQ\pi^*(\Delta_p+\Delta_q)\geq\epsilon E.    
    \end{align*}
\end{lemma}
\begin{proof} Since the construction of covering families of tigers is done via 
the Hilbert scheme, $\pi$ is finite on the general fibers $V_t$ of 
$f:Y\rightarrow B$. Recall that $\pi(V_t)\subseteq X$ is the minimal non-klt center 
of $(X,\Delta_{p(t)})$ for some $\Delta_{p(t)}$ passing through a general 
point $p(t)\in X$. We denote $\Delta_{p(t)}$ by $\Delta_t$ for simplicity.

Let $E$ be a $\pi$-exceptional divisor dominating $B$. Since $E\cap V_b$ is one dimensional 
for general $b\in B$ and $\pi|_{V_b}$ is finite, $\dim\pi(E)>0$ as $\pi(E)\supseteq\pi(E\cap V_b)$. 
Since $E$ is irreducible and $\pi$-exceptional, $\pi(E)$ is an irreducible curve. 
Fix $t_1,t_2\in B$ two general points. Pick a general point $x\in \pi(E)$ and 
consider its preimage on $V_{t_i}$. Since $\pi$ is finite on the general fiber 
$V_t$, $\pi^{-1}(x)\cap V_{t_i}$ can only be a discrete finite set. Choose 
$x_i\in \pi^{-1}(x)\cap V_{t_i}$ over $x$ for $i=1,2$. 
Apply the Connectedness Lemma \ref{con} to the pair $(Y,\Gamma-R+\pi^*(\Delta_{t_1}+\Delta_{t_2}))$ 
over $X$. There is a (possibly reducible) curve contained in 
$\pi^{-1}(x)\cap{\rm Nklt}(Y,\Gamma-R+\pi^*(\Delta_{t_1}+\Delta_{t_2}))$ 
connecting $x_1$ and $x_2$. The component of this curve containing $x_1$ 
can not lie on $V_{t_1}$ as the map $\pi$ is finite on $V_{t_1}$. As $x\in\pi(E)$ 
is general, this curve deforms into a dimension two subset of $E$ by moving 
$x\in \pi(E)$. Since $E$ is irreducible, the closure of this two dimensional 
subset coincides with $E$ and hence $E\subseteq{\rm Nklt}(K_Y+\Gamma-R+\pi^*(\Delta_{t_1}+\Delta_{t_2}))$. 
In particular, ${\rm mult}_E(K_Y+\Gamma-R+\pi^*(\Delta_{t_1}+\Delta_{t_2}))\geq1$. 
If $E\nsubseteq{\rm Supp}(\Gamma)$, then $\pi^*(\Delta_p+\Delta_q)\geq E$. 
If $E\subseteq{\rm Supp}(\Gamma)$, then $\pi^*(\Delta_p+\Delta_q)\geq\epsilon E$ 
since $\Gamma\in[0,1-\epsilon)$ as $X$ is $\epsilon$-klt.  
\end{proof}

To study the geometry of the covering family $f:Y\rightarrow B$, we would like to 
run a relative minimal model program of $(Y,\Gamma)$ over $B$. However, 
$Y$ is normal but possibly not $\QQ$-factorial. To get a $\QQ$-factorial model 
of $(Y,\Gamma)$, we adopt Hacon's dlt models, cf. \cite[Theorem 3.1]{KK:DB}. In fact, 
since the volume bound will be obtained by doing a computation on a general fiber 
$Y_b$, it suffices to modify $Y$ over an open subset $U\subseteq B$. 

\begin{lemma}\label{Tech} After restricting to an open subset $U\subseteq B$ and 
replacing $Y$ by a suitable birational model, we can assume that $Y$ is $\QQ$-factorial 
and $(Y,\Gamma)$ is $\epsilon/2$-klt. Moreover, we can assume for $E$ any 
$\pi$-exceptional divisor dominating $U$ and $p,q\in X$ general, we have that 
    \begin{equation}\label{40} \frac{2}{\omega'}H\sim_\QQ\pi^*(\Delta_p+\Delta_q)\geq\frac{\epsilon}{2}E.    
    \end{equation}
\end{lemma}

\begin{proof} Fix $p,q\in X$ general and consider the pair 
    \begin{equation} K_Y+\Gamma-R_{\rm d}+\pi^*(\Delta_p+\Delta_q)-R_{\rm e}\sim_\QQ \pi^*(K_X+\Delta_p+\Delta_q) \tag{$\sharp$}
    \end{equation}
where $R=R_{\rm d}+R_{\rm e}$ with $(-)_{\rm d}$ the sum of components dominating $B$ 
and $(-)_{\rm e}$ the sum of components mapping to points in $B$. Restricting $Y$ to 
$Y_U=f^{-1}(U)$ for a suitable nonempty open set $U\subseteq B$, we may assume that 
$R_{\rm e}=0$ and $(\sharp)$ becomes
    \begin{equation*} K_Y+\Gamma-R_{\rm d}+\pi^*(\Delta_p+\Delta_q)\sim_\QQ \pi^*(K_X+\Delta_p+\Delta_q). 
    \end{equation*}
We abuse the notation: $Y$ is understood to be $Y_U$ if not specified. 

Denote $\Gamma_{p,q}=\Gamma-R_{\rm d}+\pi^*(\Delta_p+\Delta_q)$. Note 
that $\Gamma_{p,q}\geq0$ by Lemma \ref{two}. Let $\phi:W\rightarrow Y$ 
be a log resolution of $(Y,\Gamma_{p,q})$ and write 
    \begin{align*}K_W+{\phi}^{-1}_*\Gamma_{p,q}+Q\sim_\QQ \phi^*(K_Y+\Gamma_{p,q})+P, 
    \end{align*}
where $Q,P\geq0$ are $\phi$-exceptional divisors with $Q\wedge P=0$. We 
aim to modify $W$ by running a relative minimal model program over $Y$ 
with scaling of an ample divisor so that it contracts $Q^{<1-\epsilon/2}+P$, 
where $(\sum_ia_iQ_i)^{<\alpha}:=\sum_{a_i<\alpha}a_iQ_i$. Note that we define 
$(-)^{\alpha\leq\cdot<\beta}$ and $(-)^{\geq\alpha}$ in the same way.  

Consider $F=\sum_i  F_i$, where the sum runs over all the $\phi$-exceptional 
divisors with log discrepancy in $(\epsilon/2,1]$ with respect to $(Y,\Gamma_{p,q})$, 
then 
    \begin{align*}(F+P)\wedge Q^{\geq1-\epsilon/2}=0,\ {\rm and}\ {\rm Supp}(F)\supseteq{\rm Supp}(Q^{<1-\epsilon/2}).
    \end{align*}
	%\item ${\rm Supp}(F)\cup{\rm Supp}(Q^{<1-\epsilon/2})$ is the set of all 
	%      the $\phi$-exceptional divisors with discrepancy in $(-1+\epsilon/2,0]$ 
	%      with respect to $(Y,\Gamma_{p,q})$.
Since $(Y,\Gamma-R)$ is $\epsilon$-klt, the divisor $\Gamma$ on $Y$ as 
well as $\phi^{-1}_*\Gamma$ on $W$ has coefficients in $[0,1-\epsilon)$. 
For rational numbers $0<\epsilon<\epsilon'<1$ and $0<\delta,\delta'\ll1$, 
we have the following $\epsilon/2$-klt pair
    \begin{align*} &K_W+\phi^{-1}_*\Gamma+Q^{<1-\epsilon/2}+\delta'Q^{1-\epsilon/2\leq\cdot<1}+(1-\epsilon')(Q^{\geq1})_{\rm red}+\delta F \\
	\sim_\QQ\ &\phi^*(K_Y+\Gamma_{p,q})-(\phi^{-1}_*\Gamma_{p,q}-\phi^{-1}_*\Gamma)-(1-\delta')Q^{1-\epsilon/2\leq\cdot<1}-(Q^{\geq1}-(1-\epsilon')(Q^{\geq1})_{\rm red})\\
                 &+P+\delta F   
    \end{align*}
where $(\sum_jb_jG_j)_{\rm red}:=\sum_{b_j\neq0}G_j$. We denote the above pair by $(W,\Xi)$ 
where 
    \begin{align*}\Xi=\phi^{-1}_*\Gamma+Q^{<1-\epsilon/2}+\delta'Q^{1-\epsilon/2\leq\cdot<1}+(1-\epsilon')(Q^{\geq1})_{\rm red}+\delta F.
    \end{align*}

By \cite{BCHMI}, a relative minimal model program with scaling of an ample divisor 
of the pair $(W,\Xi)$ over $Y$ terminates with a birational model $\psi:W\dashrightarrow W'$ 
over $Y$ with $\phi':W'\rightarrow Y$ the induced map. We obtain the following diagram, 
\begin{center}
\begin{tikzpicture}
\node (A) at (-2,0) {$X$};
\node (B) at (0,0){ $Y$};
\node (C) at (2,0){$Y_U$};
\node (D) at (0,-1){$B$};
\node (E) at (2,-1){$U$};
\node (G) at (2,1){$W'$};
\node (H) at (4,1){$W$};
\path[->] (B) edge node[below]{$\pi$}(A);
\path[->] (B) edge node[right]{$f$}(D);
\path[left hook->] (C) edge node[left]{}(B);
\path[->] (C) edge node[left]{}(E);
\path[left hook->] (E) edge node[left]{}(D);
\path[->] (H) edge node[right]{$\phi$}(C);
\path[dashed][->] (H) edge node[above]{$\psi$}(G);
\path[->] (G) edge node[right]{$\phi'$}(C);
\path[->] (G) edge node[above]{$\pi'$}(A);
\end{tikzpicture}
\end{center}
where $\pi':W'\rightarrow X$ is the induced map.

Write $K_{W'}+\Gamma_{W'}-R_{W'}\sim_\QQ\pi'^*K_X$ where $\pi'=\phi'\circ\pi$. 
Note that $\Gamma_{W'}\in[0,1-\epsilon)$ by the $\epsilon$-klt condition and 
$\Gamma_{W'}-(\phi')^{-1}_*\Gamma\geq0$ is $\phi'$-exceptional. It follows by the construction 
that $\Gamma_{W'}\leq\psi_*\Xi$. %\begin{align*} \psi^{-1}_*\Gamma_{W'}\leq \phi^{-1}_*\Gamma+\delta'Q^{1-\epsilon/2\leq\cdot<1}+(1-\epsilon')(Q^{\geq1})_{\rm red}.    \end{align*} 
In particular, $(W',\Gamma_{W'})$ is $\epsilon/2$-klt as the pair $(W,\Xi)$ is $\epsilon/2$-klt 
and the minimal model program does not make singularities worse.  

On $W'$, the divisor 
    \begin{align*} G=\psi_*(-(\phi^{-1}_*\Gamma_{p,q}-\phi^{-1}_*\Gamma)-(1-\delta')Q^{1-\epsilon/2\leq\cdot<1}-(Q^{\geq1}-(1-\epsilon')(Q^{\geq1})_{\rm red})
                 +P+\delta F)          
    \end{align*}
is $\phi'$-nef with $\phi'_*G\leq0$ since $\Gamma_{p,q}\geq\Gamma$. 
By \cite[Negativity Lemma 3.39]{KM:bgav}, we have that $G\leq0$. Since $F+P$ 
is \mbox{$\phi$-exceptional} and $(F+P)\wedge Q^{\geq1-\epsilon/2}=0$, 
it follows that $\psi_*(P+\delta F)=0$. In particular, all the $\phi'$-exceptional divisors 
on $W'$ have log discrepancies less than or equal to $\epsilon/2$ with respect 
to $(Y,\Gamma_{p,q})$. 

We now show that for any $\pi'$-exceptional divisor $E'$ on $W'$ dominating $U$, $E'$ satisfies 
the inequality 
    \begin{align*} \frac{2}{\omega'}H'\sim_\QQ\pi'^*(\Delta_p+\Delta_q)\geq\frac{\epsilon}{2}E',    
    \end{align*}
where $H'=\pi'^*(-K_X)$. This easy to see. If $E=\phi'_*(E')\neq0$ on $Y_U$, then 
by Lemma \ref{two}, $E\subseteq{\rm Nklt}(K_Y+\Gamma-R+\pi^*(\Delta_p+\Delta_q))$ and hence 
$E'\subseteq{\rm Nklt}(K_{W'}+\Gamma_{W'}-R_{W'}+\pi'^*(\Delta_p+\Delta_q))$. The inequality 
then follows from the same argument as in Lemma \ref{two}. If $\phi'_*E'=0$, then 
by construction ${\rm mult}_{E'}(K_{W'}+\Gamma_{W'}-R_{W'}+\pi'^*(\Delta_p+\Delta_q))\geq1-\epsilon/2$. 
Suppose that $E'\subseteq{\rm Supp}(R_{W'})$, then 
    \begin{align*} \frac{2}{\omega'}H'\sim_\QQ\pi'^*(\Delta_p+\Delta_q)\geq E'\geq\frac{\epsilon}{2}E'.    
    \end{align*}
If $E'\subseteq{\rm Supp}(\Gamma_{W'})$, then as $\Gamma_{W'}\in[0,1-\epsilon)$ we get 
    \begin{align*} \frac{2}{\omega'}H'\sim_\QQ\pi'^*(\Delta_p+\Delta_q)\geq((1-\frac{\epsilon}{2})-(1-\epsilon))E'=\frac{\epsilon}{2}E'.    
    \end{align*}

It follows that $W'$ satisfies the required properties. 
\end{proof}

\begin{remark} Write $\Gamma=\pi^{-1}_*\Delta+\Gamma_{\rm d}+\Gamma_{\rm e}$ and 
$R=R_{\rm d}+R_{\rm e}$, where $(-)_{\rm d}$ is the sum of components dominating $B$ 
and $(-)_{\rm e}$ is the sum of components mapping to points in $B$. From the proof of 
Lemma \ref{Tech}, we deduce the following two inequalities : 
    \begin{equation*}\label{41}  \frac{2}{\omega'}H\sim_\QQ \pi^*(\Delta_p+\Delta_q)\geq R_{\rm d}\ {\rm and}\
				\frac{2}{\omega'}H\sim_\QQ \pi^*(\Delta_p+\Delta_q)\geq\frac{\epsilon}{2}\Gamma_{\rm d}. \tag{5.2}
    \end{equation*}
\end{remark}

Now let $\pi:Y\rightarrow X$ with $f:Y\rightarrow U$ be the modified birational 
covering family of tigers of dimension two and weight $\omega'\geq\omega/2$ 
given by Lemma \ref{Tech}, where $Y$ is now $\QQ$-factorial. Write $K_Y+\Gamma-R\sim_\QQ \pi^*K_X$, 
where $\Gamma, R\geq0$ are $\pi$-exceptional and $\Gamma\wedge R=0$. 
The pair $(Y,\Gamma)$ is $\epsilon/2$-klt with $\Gamma\in[0,1-\epsilon/2)$ and 
note that $H=\pi^*(-K_X)$ is semi-ample and big on $Y$. 

Recall that for a projective morphism $\phi:Z\rightarrow U$, a divisor $D$ on 
$Z$ is pseudo-effective (PSEF) over $U$ if the restriction of $D$ to the generic 
fiber is pseudo-effective. 
\begin{lemma} Assume that $\omega'>2$ and consider the pseudo-effective 
threshold of $K_Y+\Gamma$ over $U$ with respect to $H$ 
    \begin{align*} \tau:=\inf\{t>0|K_Y+\Gamma+tH\ {\rm is\ PSEF\ over}\ B\}. 
    \end{align*}
Then $1\geq\tau\geq 1-\frac{2}{\omega'}>0.$
\end{lemma}
\begin{proof} Since $K_Y+\Gamma+H\sim_\QQ R\geq0$, the first inequality is clear. 
Restricting to a general fiber $Y_u$ of $Y$ over $U$, we have
       \begin{align*} (K_Y+\Gamma+\tau H)|_{Y_u}=&(R-(1-\tau)H)|_{Y_u} \\
                                                =&(R_d-\frac{2}{\omega'}H)|_{Y_u}-(1-\tau-\frac{2}{\omega'})H|_{Y_u} 
       \end{align*}
which can not be PSEF if $\omega'>2$ and $\tau<1-\frac{2}{\omega'}$ 
since the first term is non-positive by \eqref{41} and the second term is negative. 
\end{proof}

Now we run a relative minimal model program with scaling for the covering 
family of tigers $f:Y\rightarrow U$. Since $(Y,\Gamma)$ is $\epsilon/2$-klt 
and $H$ is semiample and big, we may assume that $(Y,\Gamma+\tau' H)$ 
remains $\epsilon/2$-klt for any rational number $0<\tau'<\tau$. By \cite{BCHMI}, a 
relative minimal model program of $(K_Y+\Gamma+\tau'H)$ with scaling of $H$ 
over $U$ terminates with a relative Mori fiber space $Y'\rightarrow T$ over $U$ 
with $\dim Y'>\dim T\geq\dim U$. Denote the induced 
maps by $g:Y\dashrightarrow Y'$, $\psi:Y'\rightarrow T$, and $\phi:Y'\rightarrow U$. 
We obtain the following diagram, 
\begin{center}
\begin{tikzpicture}
\node (A) at (-2,1) {$X$};
\node (B) at (0,1) {$Y$};
\node (D) at (2,1) {$Y'$};
\node (C) at (0,-1){$U$};
\node (E) at (2,-1){$T.$};
\node (F) at (-0.60,-1){$B\supseteq$};
\path[->] (B) edge node[above]{$\pi$}(A);
\path[->] (B) edge node[left]{$f$}(C);
\path[dashed][->] (B) edge node[above]{$g$}(D);
\path[->] (D) edge node[left]{$\phi$}(C);
\path[->] (D) edge node[right]{$\psi$}(E);
\path[->] (E) edge node[left]{}(C);
\end{tikzpicture}
\end{center}
For a general fiber $Y'_t$ of $\psi:Y'\rightarrow T$, by construction, the Picard number 
$\rho(Y'_t)=1$ and the divisor $-(K_{Y'}+\Gamma'_d)|_{Y'_t}\sim_\QQ (H'-R_d)|_{Y'_t}$ on $Y'_t$
is ample. 

\begin{lemma}\label{Exc2} There exists a divisor $E'$ on $Y'$ which is 
exceptional over $X$ and dominates $T$. %In fact, any such a divisor is the proper transform of a divisor on $Y$ which is $\pi$-exceptional.
\end{lemma}
\begin{proof} Recall that there is a natural map $T\rightarrow U\rightarrow B$. 
We can extend $\psi:Y'\rightarrow T$ to $\overline{\psi}:\overline{Y'}\rightarrow\overline{T}$ 
over $B$ where $\overline{(-)}$ stands for a projective compactification of $(-)$. 
Take a common resolution $p:W\rightarrow X$ and $q:W\rightarrow\overline{Y'}$ 
and let $A_{\overline{T}}$ be a sufficiently ample divisor on $\overline{T}$. Let 
$A_{\overline{Y'}}=\overline{\psi}^*A_{\overline{T}}$, $A_W=q^*A_{\overline{Y'}}$, 
and $A_X=p_*A_W$. Then $p^*A_X=A_W+E=q^*A_{\overline{Y'}}+E=q^*\overline{\psi}^*A_{\overline{T}}+E$ 
for an effective divisor $E$ on $W$ which is exceptional over $X$. Since $\rho(X)=1$, 
it follows by the same argument as in Lemma \ref{exc} that one of the 
irreducible components of $E$ maps to a divisor $E'$ on $\overline{Y'}$. 
By the same argument as in Lemma \ref{exc} again, one of the irreducible 
components of the nonzero divisor $q_*(E)$ dominates $\overline{T}$. %The last statement follows from the fact that a minimal model program does 
%not extract divisors. 
\end{proof}

\begin{proposition}\label{21} If $\dim T=2$, then $\omega'\leq 8/\epsilon+2$.
\end{proposition}
\begin{proof} By Lemma \ref{Exc2}, there exists a divisor $E'$ on $Y'$ which 
is exceptional over $X$ and dominates $T$. Note that $Y'$ is normal and hence 
$\psi({\rm Sing}(Y'))$ is a proper subset of $T$. In particular, a general 
fiber $Y'_t$ of $\psi:Y'\rightarrow T$ is a smooth projective curve and hence 
$E'.Y'_t\geq1$. Since the divisor $-(K_{Y'}+\Gamma'_d)|_{Y'_t}\sim_\QQ (H'-R_d)|_{Y'_t}$ 
is ample, a general fiber $Y'_t$ is a smooth rational curve $\mathbb{P}^1$. 
From \eqref{40}, we know that  
    \begin{align*} \frac{2}{\omega'}H'-\frac{\epsilon}{2}E'\sim_\QQ {\rm effective}. 
    \end{align*}
Also from \eqref{41},
    \begin{align*} -(K_{Y'}+\Gamma').Y'_t=(H'-R').Y'_t=&(1-\frac{2}{\omega'})H'.Y'_t+(\frac{2}{\omega'}H-R').Y'_t \\
                                                      \geq&(1-\frac{2}{\omega'})H'.Y'_t.  
    \end{align*} 
It follows that  
    \begin{align*} 
      \frac{2}{\omega'}\geq\frac{1}{\omega'}(-(K_{Y'}+\Gamma').Y'_t)&\geq\frac{1}{\omega'}(1-\frac{2}{\omega'})H'.Y'_t \\
                                                                   &\geq(1-\frac{2}{\omega'})\frac{\epsilon}{4}E'.Y'_t \\
                                                                   &\geq(1-\frac{2}{\omega'})\frac{\epsilon}{4}
    \end{align*}
where the first inequality follows by the adjunction formula on $\mathbb{P}^1$. 
Hence $\omega'\leq \frac{8}{\epsilon}+2.$
\end{proof}

\begin{proposition}\label{22} If $\dim T=1$, then 
    \begin{align*} \omega'\leq\frac{4M(2,\epsilon)R(2,\epsilon)}{\epsilon}+2 
    \end{align*} 
where $R(2,\epsilon)$ is an upper bound of the Cartier index of $K_S$ for $S$ any $\epsilon/2$-klt log del Pezzo surface of $\rho(S)=1$ 
and $M(2,\epsilon)$ is an upper bound of the volume ${\rm Vol}(-K_S)=K_S^2$ for $S$ any $\epsilon/2$-klt log del Pezzo surface of $\rho(S)=1$. 
\end{proposition}
\begin{proof} Since $f:Y\rightarrow U$ has connected fibers, $T\cong U$. 
Since $-(K_{Y'}+\Gamma'_d)|_{Y'_u}\sim_\QQ (H'-R_d)|_{Y'_u}$ is ample 
and $\rho(Y'_u)=1$ for a general point $u\in U$, we see that
    \begin{align*} -K_{Y'_u}\sim_\QQ (H'+\Gamma'_d-R_d)|_{Y'_u}    
    \end{align*}
is ample. By Lemma \ref{Exc2}, let $E'$ be a divisor on $Y'$ exceptional 
over $X$ which dominates $U$, then  
    \begin{align*} -K_{Y'_u}\equiv (H'+\Gamma'_d-R_d)|_{Y'_u}\geq(1-\frac{2}{\omega'})H|_{Y'_u}
                   \geq(1-\frac{2}{\omega'})\cdot\frac{\omega'\epsilon}{4}E'_u
    \end{align*}
where the second inequality follows by dropping $\Gamma'_d$ and applying 
\eqref{41} while the last one from \eqref{40}. By intersecting with the ample 
divisor $-K_{Y'_u}$, this implies that 
    \begin{align*} (-K_{Y'_u})^2\geq(\omega'-2)\frac{\epsilon}{4}E'_u.(-K_{Y'_u}). 
    \end{align*}
Now $(Y'_u, \Gamma'_u)$ is an $\epsilon/2$-klt log del-Pezzo surfaces of Picard 
number one. Hence $Y'_u$ is an $\epsilon/2$-klt del-Pezzo surface 
of Picard number $\rho(Y'_u)=1$. By Theorem \ref{I}, $(-K_{Y'_u})^2$ 
is bounded above by a positive number $M(2,\epsilon)$ satisfying 
    \begin{align*} M(2,\epsilon)\leq\max\{64,\frac{16}{\epsilon}+4\}.     
    \end{align*}
Also, by $(\Diamond)$ the Cartier index of $K_{Y'_u}$ has an upper bounded 
    \begin{align*} R(2,\epsilon)\leq r(2,\frac{\epsilon}{2})\leq 2(4/\epsilon)^{128\cdot2^5/\epsilon^5}.
    \end{align*} 
It follows that 
    \begin{align*} M(2,\epsilon)\geq(-K_{Y'_u})^2\geq\frac{1}{R(2,\epsilon)}(\omega'-2)\frac{\epsilon}{4}E'_u.({\rm Ample\ Cartier})
		  \geq\frac{1}{R(2,\epsilon)}(\omega'-2)\frac{\epsilon}{4}
    \end{align*}
and hence we get an upper bound
    \begin{align*} \omega'\leq\frac{4M(2,\epsilon)R(2,\epsilon)}{\epsilon}+2.
    \end{align*} 
\end{proof}

\begin{remark} It has been shown in \cite{Bel} that a klt log del Pezzo 
surface has at most four isolated singularities. Also surface klt singularities 
are classified by Alexeev in \cite{fa3}. Hence we expect that it 
is possible to obtain a better upper bound for $R(2,\epsilon)$ and $M(2,\epsilon)$ 
in Proposition \ref{22}. 
\end{remark}

\begin{theorem}\label{II} Let $(X,\Delta)$ be an $\epsilon$-klt log 
$\QQ$-Fano threefold of $\rho(X)=1$. Then the degree $-K_X^3$ satisfies
    \begin{align*} -K_X^3\leq(\frac{24M(2,\epsilon)R(2,\epsilon)}{\epsilon}+12)^3 
    \end{align*}
where $R(2,\epsilon)$ is an upper bound of the Cartier index of $K_S$ for $S$ any $\epsilon/2$-klt log del Pezzo surface of $\rho(S)=1$ 
and $M(2,\epsilon)$ is an upper bound of the volume ${\rm Vol}(S)=K_S^2$ for $S$ any $\epsilon/2$-klt log del Pezzo surface of $\rho(S)=1$. 
Note that we have $M(2,\epsilon)\leq\max\{64,16/\epsilon+4\}$ from Theorem \ref{I} 
and $R(2,\epsilon)\leq2(4/\epsilon)^{128\cdot2^5/\epsilon^5}$ from $(\Diamond)$. 
\end{theorem}
\begin{proof} Recall that $\omega'\geq\omega/2$. The theorem then follows 
from Propositions \ref{one}, \ref{21} and \ref{22}.
\end{proof}

The following example shows that the cone construction analogous to Example 
\ref{guiding} only provides $\epsilon$-klt Fano threefolds with volumes of order 
$1/\epsilon^2$. 
\begin{example}\label{cone'}(Projective cone of projective spaces) For 
$n\geq1$ and $d\geq2$, let $\mathbb{P}^n\hookrightarrow\mathbb{P}^N$ be 
the embedding by $|\mathcal{O}(d)|$ and $X$ be the associated projective 
cone. The projective variety $X$ is normal $\QQ$-factorial of Picard number 
one with unique singularity at the vertex $O$. Also, $X$ admits a resolution 
$\pi:Y=Bl_OX\rightarrow X$ with exceptional divisor $E\cong\mathbb{P}^n$ of 
normal bundle $\mathcal{O}_E(E)\cong\mathcal{O}_{\mathbb{P}^n}(-d)$. The 
variety $Y$ is the projective bundle 
$\mu:Y\cong\mathbb{P}_{\mathbb{P}^n}(\mathcal{O}_{\mathbb{P}^n}\oplus\mathcal{O}_{\mathbb{P}^n}(-d))\rightarrow\mathbb{P}^n$ 
with tautological bundle $\mathcal{O}_Y(1)\cong\mathcal{O}_Y(E)$. We have:
    \begin{itemize}
      \item $\mathcal{O}_E(E)\cong\mathcal{O}_{\mathbb{P}^n}(-d)$ and 
	    hence $E^{n+1}=(-d)^n$;
      \item $K_Y=\pi^*K_X+(-1+\frac{n+1}{d})E$ and hence $X$ is always klt. 
	    Also, $X$ is terminal (resp. canonical) if and only if 
            $n+1>d\geq2$ (resp. $n+1\geq d\geq2$);
      \item $K_Y=\mu^*(K_{\mathbb{P}^n}+\det(\mathcal{E}))\otimes\mathcal{O}_Y(-{\rm rk}(\mathcal{E}))\equiv-(n+1+d)F-2E$ 
	    where the vector bundle $\mathcal{E}=\mathcal{O}_{\mathbb{P}^n}\oplus\mathcal{O}_{\mathbb{P}^n}(-d)$ 
	    and $F=\mu^*\mathcal{O}_\mathbb{P}^n(1)$;
      \item $F^{n+1}=0$ and $F^{n+1-k}.E^{k}=(-d)^{k-1}$ for $1\leq k\leq n+1$;
      \item $K_Y^{n+1}=K_X^{n+1}+(-1+\frac{n+1}{d})^{n+1}E^{n+1}$ and
            \begin{align*} K_Y^{n+1}=&\frac{-1}{d}\sum_{k=1}^{n+1}\left(\begin{array}{c} n+1-k \\ k \end{array}\right)(-1+\frac{n+1}{d})^{n+1-k}(2d)^k\\
                                    =&\frac{-1}{d}((d-n-1)^{n+1}-(-(d+n+1)^{n+1}));
            \end{align*} 
      \item In summary, $-K_X$ is ample with 
            \begin{align*} (-K_X)^{n+1}=\frac{(d+n+1)^{n+1}}{d}.             
            \end{align*}
    \end{itemize}
If $n=2$, then we have an $\epsilon$-klt Fano threefold of Picard number 
one with $\epsilon=1/d$. The volume ${\rm Vol}(X)=(-K_X)^3$ is of order 
$1/\epsilon^2$.
\end{example}

In view of Theorem \ref{II}, it is then interesting to see whether $\epsilon$-klt 
Fano threefolds with big volumes exist.
\begin{qn} Can one find $\epsilon$-klt $\QQ$-factorial $\QQ$-Fano threefolds $X$ 
of $\rho(X)=1$ with volume ${\rm Vol}(X)=(-K_X)^3=O(\frac{1}{\epsilon^c})$ for $c\geq3$?  
\end{qn}
\bibliographystyle{siam}
\bibliography{BAB}
\end{document}